\documentclass{article}

\usepackage{amsmath,amsfonts,latexsym,amssymb,amsthm,url}

\usepackage[utf8]{inputenc}

    \makeatletter
    \let\@fnsymbol\@alph
    \makeatother

\newcommand{\C}{{\mathbb C}}
\newcommand{\Q}{{\mathbb Q}}
\newcommand{\R}{{\mathbb R}}
\newcommand{\Z}{{\mathbb Z}}

\newcommand{\Gal}{\mathrm{Gal}}
\newcommand{\height}{\mathrm{h}}
\newcommand{\id}{\mathrm{id}}
\renewcommand{\Im}{\mathrm{Im}}
\renewcommand{\Re}{\mathrm{Re}}

\newcommand{\discr}{\mathcal{D}}
\newcommand{\calF}{\mathcal{F}}

\newcommand{\norm}{\mathcal{N}}

\newcommand\eps\varepsilon
\newcommand\ph\varphi

\newcommand{\tilE}{{\widetilde E}}

\newcommand{\gerp}{\mathfrak{p}}

\newtheorem{theorem}{Theorem}[section]

\newtheorem{proposition}[theorem]{Proposition}

\newtheorem{corollary}[theorem]{Corollary}

\newtheorem{lemma}[theorem]{Lemma}

\newtheorem{remark}[theorem]{Remark}

\newtheorem{problem}[theorem]{Problem}

\numberwithin{equation}{section}

\renewcommand{\mod}{\bmod\,}

\title{Trinomials, singular moduli and Riffaut's conjecture}

\author{Yuri Bilu\footnote{Institut de Mathématiques de Bordeaux, Université de Bordeaux \& CNRS; partially supported by the MEC CONICYT Project PAI80160038 (Chile), and by the SPARC Project P445 (India)}, 
Florian Luca\footnote{School of Mathematics, University of the Witwatersrand;
Research Group in Algebraic Structures and Applications, King Abdulaziz University, Jeddah; 
Centro de Ciencias Matematicas, UNAM, Morelia; partially supported by the CNRS ``Postes rouges'' program},
Amalia Pizarro-Madariaga\footnote{Instituto de Matemáticas, Universidad de Valparaíso; partially supported by the Ecos / CONICYT project ECOS170022  and by the MATH-AmSud project NT-ACRT 20-MATH-06}}

\makeatletter

\renewcommand*\l@section[2]{%
  \ifnum \c@tocdepth >\z@
    \addpenalty\@secpenalty
    \addvspace{0.2em \@plus\p@}%
    \setlength\@tempdima{1.5em}%
    \begingroup
      \parindent \z@ \rightskip \@pnumwidth
      \parfillskip -\@pnumwidth
      \leavevmode \bfseries
      \advance\leftskip\@tempdima
      \hskip -\leftskip
      #1\nobreak\hfil \nobreak\hb@xt@\@pnumwidth{\hss #2}\par
    \endgroup
  \fi}

\makeatother

\begin{document}

\hfuzz 5pt
\maketitle

\begin{abstract}
Riffaut~\cite{Ri19} conjectured that a singular modulus of degree ${h\ge 3}$ cannot be a root of a trinomial with rational coefficients. We show that this conjecture follows from the GRH and obtain partial unconditional results. 
\end{abstract}

{\footnotesize

\tableofcontents

}

\section{Introduction}
A \textit{singular modulus} is the $j$-invariant of an elliptic curve with complex multiplication. Given a singular modulus~$x$ we denote by $\Delta_x$ the discriminant of the associated imaginary quadratic order. 
We denote by $h(\Delta)$ the class number of the imaginary quadratic order of discriminant~$\Delta$. Recall that two singular moduli~$x$ and~$y$ are conjugate over~$\Q$ if and only if ${\Delta_x=\Delta_y}$, and that there are $h(\Delta)$ singular moduli of a given discriminant~$\Delta$.   In particular, ${[\Q(x):\Q]=h(\Delta_x)}$. For all details, see, for instance, \cite[\S 7 and \S 11]{Co13}.

In this article we study the following question. 

\begin{problem}
Can a singular modulus of degree ${h\ge3}$ be a root of a trinomial with rational coefficients? 
\end{problem}

Here and below \textit{a trinomial} is an abbreviation for \textit{a monic trinomial non-vanishing at~$0$}; in other words,  a polynomial of the form ${t^m+At^n+B}$ with ${m>n>0}$ and  ${B\ne 0}$.

This problem emerged in the work of Riffaut~\cite{Ri19}  on the effective André-Oort conjecture. We invite the reader to consult the article of Riffaut for more context and motivation. Riffaut conjectured that the answer is negative, but, as he admits, ``much about trinomials is known, but this knowledge is still insufficient to rule out
such a possibility''. 

We believe, however, that the problem is motivated on its own, independently of Riffaut's work, because it is  of interest to learn more about the relationship between two very classical objects like rational trinomials and singular moduli. 


In Section~\ref{sproofthgrh}  we prove that Riffaut's conjecture follows from the GRH.  

\begin{theorem}
\label{thgrhintro}
Assume the Generalized Riemann Hypothesis for the  Dirichlet  $L$-functions. Then a singular modulus of degree at least~$3$ cannot be a root of a trinomial with rational coefficients. 
\end{theorem}

We also obtain some partial unconditional results. To state them, we have to introduce some definitions that will be used throughout the article. 

Let~$\Delta$ be 
 an imaginary quadratic discriminant. We call~$\Delta$ \textit{trinomial discriminant} if ${h(\Delta)\ge 3}$ and the singular moduli of discriminant~$\Delta$ are roots of a trinomial with rational coefficients. If this trinomial is of the form ${t^m+At^n+B}$ then we say that~$\Delta$ is a trinomial discriminant of \textit{signature} $(m,n)$.

Note that a trinomial discriminant may, a priori, admit several signatures. However, there can be at most finitely many of them, and all of them can be effectively computed in terms of~$\Delta$. This follows from the results of  article~\cite{BL20} and the following property: for any singular modulus~$x$ and positive integer~$k$ we have ${\Q(x^k)=\Q(x)}$, see \cite[Lemma~2.6]{Ri19}. 

Now we are ready to state our unconditional results.
First of all, in Sections~\ref{slittle} and~\ref{sallbutone} we show that a trinomial discriminant cannot be too small, and, with at most one exception, cannot be too large either.

\begin{theorem}
\label{thbounds}
Every trinomial discriminant~$\Delta$ satisfies ${|\Delta|>10^{11}}$, and at most one trinomial discriminant~$\Delta$ satisfies ${|\Delta|\ge 10^{160}}$. In particular, the set of trinomial discriminants is finite. 
\end{theorem}


Next, in Section~\ref{sstruc} we show that trinomial discriminants are of rather special form.  

\begin{theorem}
\label{thfundintro}
Every trinomial discriminant is of the form $-p$ or $-pq$, where~$p$ and~$q$ are distinct odd prime numbers. In particular, trinomial discriminants are odd and fundamental. 
\end{theorem}

Finally, in Section~\ref{ssign} we show that trinomials vanishing at singular moduli are themselves quite special. 

\begin{theorem}
\label{thsign}
Let~$\Delta$ be a trinomial discriminant of signature $(m,n)$. Assume that ${|\Delta|\ge 10^{40}}$. Then ${m-n\le 2}$. 
\end{theorem}

\paragraph{Plan of the article.}
In Section~\ref{sdom} we remind general facts about singular moduli, to be used throughout the article. In Section~\ref{ssuitin} we introduce and study the basic notion of \textit{suitable integer}. A positive integer~$a$ is called \textit{suitable} for a discriminant~$\Delta$ if there exists ${b\in \Z}$ such that ${b^2\equiv\Delta\mod4a}$ and ${(b+\sqrt\Delta)/2a}$ belongs to the standard fundamental domain (plus a certain coprimality condition must be satisfied). We give various recipes for detecting suitable integers of arbitrary discriminants, so far without any reference to trinomials.

In Section~\ref{sroots} we obtain some metrical properties of roots of trinomials, both in the complex and non-archimedean setting. Applying them to singular moduli that are roots of a trinomial, we obtain the ``principal inequality'', a basic tool  instrumental for the rest of the article. In Section~\ref{ssuittrin} we use the ``principal inequality'' to study suitable integers of trinomial discriminants: they turn out to be very large, of order of magnitude ${|\Delta|^{1/2}/\log|\Delta|}$, and densely spaced. 

In Section~\ref{slittle} we show that trinomial discriminants cannot be too small (the first statement  of Theorem~\ref{thbounds}). The argument uses the results of the previous sections   and computations with  \textsf{PARI}~\cite{pari} and \textsf{SAGE}~\cite{sagemath}. 

In Section~\ref{sstruc} we prove Theorem~\ref{thfundintro} on the structure of trinomial discriminants, using careful analysis of suitable integers. In the follow-up  Section~\ref{sprimal} we show that suitable integers of trinomial discriminants are prime numbers.

In Section~\ref{sproofthgrh} we obtain the conditional result  (Theorem~\ref{thgrhintro}) and in Section~\ref{sallbutone} we obtain an unconditional upper bound for all but one trinomial discriminant (the second statement of Theorem~\ref{thbounds}). The principal arguments of these sections already appeared elsewhere~\cite{IK04,LLS15,Po17}, and we only had to adapt them to our situation. 

In Section~\ref{shrhon} we study the class number and other numerical characteristics of trinomial discriminants. Using the results from that section, we prove Theorem~\ref{thsign} in Section~\ref{ssign}.

\paragraph{Acknowledgments.}
We thank Michael Filaseta, Andrew Granville, Sanoli Gun, Tanmay Khale, Chazad Movahhedi, Olivier Ramaré, Igor Shparlinski and the \textsf{mathoverflow} user \textit{Lucia} for helpful suggestions.

We are most grateful to the anonymous referees, who corrected several mistakes, and made many very useful comments that helped us to improve the presentation. 

All calculations were performed using \textsf{PARI}~\cite{pari} or \textsf{SAGE}~\cite{sagemath}.  We thank Bill Allombert and Karim Belabas for the \textsf{PARI} tutorial. The reader may consult \url{https://github.com/yuribilu/trinomials} to view the \textsf{PARI} script used for this article.

Yuri Bilu thanks the University of Valparaiso and the Institute of Mathematical Sciences (Chennai) for stimulating working conditions.

Florian Luca worked on this project during visits to the Institute of Mathematics of Bordeaux from March to June 2019, and to the Max Planck Institute for Mathematics in Bonn from September 2019 to February 2020. He thanks these institutions for  hospitality
and support.

Amalia Pizarro-Madariaga thanks the  Max Planck Institute for Mathematics in Bonn for hospitality and  stimulating working conditions.


\subsection{Some general conventions}
\label{ssconv}
Throughout the article we use $O_1(\cdot)$ as a quantitative version of the familiar $O(\cdot)$ notation: ${X=O_1(Y)}$ means that ${|X|\le Y}$. 

We denote by $(\Delta/p)$ the Kronecker symbol.  The general definition of  the Kronecker symbol can be found, for instance, on page~202 of~\cite{IR90}. In this article, however, we will always use it in the special case when~$\Delta$ is a discriminant (and, in particular, ${\Delta\equiv 0,1\bmod 4}$) and~$p$ is a prime number. In this case $(\Delta/p)$ is just the Legendre symbol $\bmod p$ if~$p$  is an odd prime, and 
$$
\left(\frac{\Delta}{2}\right)=
\begin{cases}
1, & \Delta\equiv 1\bmod 8, \\
-1,&\Delta\equiv  5 \bmod 8,\\
0, & \Delta\equiv 0\bmod 4. 
\end{cases}
$$

We use the standard notation $\omega(\cdot)$ 
for the number of prime divisors (counted without 
multiplicities). If~$\gerp$ is a prime of a number field then we denote  ${\nu_\gerp(\cdot)}$  the $\gerp$-adic valuation (normalized so that its group of values is~$\Z$).  

\section{Generalities on singular moduli}
\label{sdom}
In this section we summarize some properties of singular moduli used in the article. 
Unless the contrary is stated explicitly, everywhere below the letter~$\Delta$ stands for an \textit{imaginary quadratic discriminant}; that is, ${\Delta<0}$ and  ${\Delta\equiv 0,1\bmod 4}$.


Denote by~$\calF$ the standard fundamental domain: the open hyperbolic triangle with vertices 
$$
\zeta_3=\frac{-1+\sqrt{-3}}{2}, \qquad \zeta_6=\frac{1+\sqrt{-3}}{2}, \qquad i\infty,
$$
together with the geodesics ${[i,\zeta_6]}$ and ${[\zeta_6,\infty]}$. 
It is well-known (see, for instance, \cite[Proposition~2.5]{BLP16} and the references therein) that 
there is a one-to-one correspondence between the singular moduli of discriminant~$\Delta$ and the set~$T_\Delta$ of triples $(a,b,c)$ of integers with ${\gcd(a,b,c)=1}$, satisfying ${b^2-4ac=\Delta}$ and 
\begin{equation*}
\text{either\quad $-a < b \le a < c$\quad or\quad $0 \le b \le a = c$}. 
\end{equation*} 
If  ${(a,b,c)\in T_\Delta}$ then ${(b+\sqrt{\Delta})/2a}$ belongs to~$\calF$, and the corresponding singular modulus is 
${j((b+\sqrt{\Delta})/2a)}$.

We call a singular modulus  \textit{dominant} if in the corresponding triple $(a,b,c)$ we have ${a=1}$. For every~$\Delta$ there exists exactly one dominant singular modulus of discriminant~$\Delta$.

The inequality 
$$
\bigl||j(z)|-e^{2\pi \Im z}\bigr|\le 2079,
$$
holds true for every ${z\in \calF}$; see, for instance, \cite[Lemma~1]{BMZ13}. In particular, if~$x$ is a singular modulus of discriminant~$\Delta$ corresponding to the triple ${(a,b,c)\in T_{\Delta}}$  then 
$$
\bigl||x|-e^{\pi|\Delta|^{1/2}/a}\bigr|\le 2079.
$$
This implies that
\begin{align}
\label{euniv}
|x|&\le e^{\pi|\Delta|^{1/2}}+ 2079&& \text{in any case};\\
\label{eifdom}
|x|&\ge e^{\pi|\Delta|^{1/2}}- 2079&& \text{if $x$ is dominant};\\
\label{eifnotdom}
|x|&\le e^{\pi|\Delta|^{1/2}/2}+ 2079&& \text{if $x$ is not dominant}. 
\end{align}
These inequalities will be systematically used in the sequel, sometimes without special reference.

\section{Suitable integers} 
\label{ssuitin}
Everywhere in this section~$\Delta$ is an imaginary quadratic discriminant and~$a$ a positive integer.

Call an integer~$a$ \textit{suitable} for~$\Delta$ if there exist ${b,c\in \Z}$ such that ${(a,b,c) \in T_\Delta}$. 
Note that~$1$ is always suitable, and that a suitable~$a$ satisfies  ${|\Delta|\ge 3a^2}$: this follows from the fact that  ${(b+\sqrt{\Delta})/2a}$ belongs to the standard fundamental domain, or directly from the relation ${\Delta=b^2-4ac}$ and the inequalities ${|b|\le a\le c}$. Moreover, equality ${|\Delta|= 3a^2}$ is possible only when ${\Delta=-3}$ and ${a=1}$, and we have the strict inequality ${|\Delta|> 3a^2}$ when ${\Delta\ne -3}$. 



In the following proposition we collect some useful tools for detecting suitable integers.

\begin{proposition}
\label{psuit}
\begin{enumerate}

\item
\label{isquare}
Assume that ${\gcd(a,\Delta)=1}$, that~$\Delta$ is a square ${\mod 4a}$, and that ${|\Delta|\ge 4a^2}$. Then~$a$ is suitable for~$\Delta$. 

\item
\label{idiv}
Let~$a$ be  suitable for~$\Delta$ and~$a'$ a divisor of~$a$ such that ${\gcd(a',\Delta)=1}$. Then~$a'$ is suitable for~$\Delta$ as well.

\item
\label{ikrone}
Let~$p$ be a prime number satisfying ${(\Delta/p)=1}$ and ${|\Delta|\ge 4p^2}$. Then~$p$ is suitable for~$\Delta$.

\item
\label{ieven}
Assume that~$\Delta$ is even, that ${\Delta\not\equiv 4\mod32}$, and that ${|\Delta|\ge76}$. Then~$2$ or~$4$ is suitable for~$\Delta$.

\item
\label{ihensel}
Assume that ${\Delta\equiv4\mod 32}$. Let ${k\ge 3}$ be an integer such that ${|\Delta| \ge 2^{2k+2}}$. Then $2^k$ is suitable for~$\Delta$. In particular, if ${|\Delta|\ge 2^{10}}$ then~$8$ and~$16$ are suitable for~$\Delta$.

\item
\label{icoprime}
Assume that ${\Delta=-2^\nu aa'}$, where ${\nu=\nu_2(\Delta)}$ and  ${a,a'}$ are  positive odd integers with ${\gcd(a,a')=1}$. Then ${\min \{a,a',(a+a')/4\}}$ is suitable for~$\Delta$ if~$\Delta$ is odd, and $\min \{a,a'\}$ is suitable  if~$\Delta$ is even.

\end{enumerate}
\end{proposition}

The proof requires a simple lemma, telling that $0^2,1^2,\ldots,m^2$ exhaust all squares $\mod 4m$. 

\begin{lemma}
\label{lfourm}
Let~$m$ be a positive integer and~$x$ an integer. Then there exists an integer~$y$ satisfying 
${0\le y\le m}$ and  ${y^2\equiv x^2 \mod 4m}$. 
\end{lemma}

\begin{proof}
Since ${x_1\equiv x_2\mod 2m}$ implies that ${x_1^2\equiv  x_2^2\mod 4m}$ we may assume that ${-m< x \le m}$. Now set ${y=|x|}$. 
\end{proof}

\begin{proof}[Proof of Proposition~\ref{psuit}]
\begin{enumerate}
\item
If~$\Delta$ is a square $\mod 4a$ then Lemma~\ref{lfourm} produces ${b\in \Z}$ satisfying ${0\le b\le a}$ and ${\Delta\equiv b^2\mod4a}$. If  ${\gcd(a,\Delta)=1}$ then we have ${\gcd(a,b)=1}$, and  ${(a,b,(b^2-\Delta)/4a)\in T_\Delta}$ when ${|\Delta|\ge 4a^2}$. This proves item~\ref{isquare}.

\item
If~$a$ is suitable for~$\Delta$ then ${|\Delta|\ge 3a^2}$ and~$\Delta$ is a square ${\mod 4a}$. If~$a'$ is a proper divisor of~$a$ then ${|\Delta|\ge 12(a')^2}$ and~$\Delta$ is a square ${\mod 4a'}$. Hence item~\ref{idiv} follows from item~\ref{isquare}.

\item
Let~$p$ a prime number (${p=2}$ included). Then the condition ${(\Delta/p)=1}$ implies that~$\Delta$ is a square ${\mod 4p}$ and is co-prime with~$p$. Hence item~\ref{ikrone} follows from item~\ref{isquare} as well. 

\item
When ${16\mid \Delta}$, select ${b\in \{0,4\}}$ to satisfy  ${\nu_2(\Delta-b^2)=4}$. Then we have  ${\bigl(4,b,(b^2-\Delta)/16\bigr)\in T_\Delta}$  provided that ${|\Delta|\ge 64}$. Furthermore,
\begin{align*}
(2,0,-\Delta/8)&\in T_\Delta && \text{if ${\Delta\equiv 8\mod16}$  and  ${|\Delta|\ge 24 }$},\\
(2,2,(4-\Delta)/8)&\in T_\Delta && \text{if ${\Delta\equiv12\mod16}$  and ${|\Delta|\ge 20}$},\\
(4,2,(4-\Delta)/16)&\in T_\Delta && \text{if ${\Delta\equiv 20\mod32}$ and ${|\Delta|\ge 76}$}. 
\end{align*} 
This proves item~\ref{ieven}.

\item
If ${\Delta\equiv 4\mod 32}$ then, by Hensel's lemma,  for ${k=3,4,\ldots}$ there exists ${x_k\in \Z}$ such that ${\Delta/4\equiv x_k^2\mod 2^{k}}$.  Hence, setting ${b_k=2x_k}$, we find ${b_k\in \Z}$  with the property   ${\Delta\equiv b_k^2\mod 2^{k+2}}$. Moreover, Lemma~\ref{lfourm} implies that~$b_k$ can be chosen to satisfy ${0\le b_k\le 2^{k}}$. Note also that ${\nu_2(b_k)=1}$, which implies that ${\nu_2(b_k^2-(2^k-b_k)^2)=k+2}$.   Hence, replacing (if necessary)~$b_k$ by ${2^k-b_k}$ we may assume  that ${\nu_2(b_k^2-\Delta) =k+2}$. Then ${(2^k,b_k, (b_k^2-\Delta)/2^{k+2})\in T_\Delta}$ provided ${|\Delta|\ge 2^{2k+2}}$. This proves item~\ref{ihensel}. 

\item

In item~\ref{icoprime} we will assume that ${a'\ge a}$. Then for odd~$\Delta$ we have
\begin{align*}
\left(a,a,\frac{a+a'}4\right)&\in T_\Delta&& \text{when $a'\ge 3a$},\\
\left(\frac{a+a'}4,\frac{a'-a}2,\frac{a+a'}4\right)&\in T_\Delta&& \text{when $a\le a'\le 3a$}, 
\end{align*}
and for even~$\Delta$ 
 we have ${(a,0,2^{\nu-2}a')\in T_\Delta}$. This proves item~\ref{icoprime}. \qed
\end{enumerate} \renewcommand{\qedsymbol}{}
\end{proof}

We want to extend item~\ref{icoprime}  to the case when~$a$ and $\Delta/a$ are not coprime. This is possible in the special case when~$a$ is an odd prime power with even exponent. 

\begin{proposition}
\label{pprimepower}
Let~$p$ be an odd prime number and~$k$ a positive integer such that ${p^{2k+1}\mid \Delta}$. Write ${\Delta=-p^{2k}m}$. (In particular, ${p\mid m}$.) 
Define
\begin{equation*}
\delta=
\begin{cases}
2,&\text{$\Delta$ is odd},\\
1, &\text{$\Delta$ is even}. 
\end{cases}
\end{equation*}
Assume that ${m\ge (4/9)p^{2k}}$. Then ${\min \bigl\{p^{2k}, (m+(p^k-\delta)^2)/4\bigr\}}$
is suitable for~$\Delta$.
\end{proposition}

\begin{proof}
Setting
\begin{equation*}
A_0=\frac{p^{2k}+2\delta p^k-3\delta^2}3, \qquad  A_1=3p^{2k}-2\delta p^k-\delta^2, \qquad A_2=3p^{2k}+2\delta p^k-\delta^2, 
\end{equation*} 
a  routine verification shows that 
\begin{align}
\left(p^{2k}, p^{2k}-\delta p^k,\frac{m+(p^k-\delta)^2}4\right)&\in T_\Delta && \text{when $m\ge A_2$},\\
\left(\frac{m+(p^k-\delta)^2}4, p^{2k}-\delta p^k,p^{2k}\right)&\in T_\Delta && \text{when $A_1\le m\le A_2$},\\
\label{ethirdcase}
\left(\frac{m+(p^k-\delta)^2}4, \frac{|m-p^{2k}+\delta^2|}2,\frac{m+(p^k+\delta)^2}4\right)&\in T_\Delta && \text{when $A_0\le m\le A_1$},\\
\frac{m+(p^k-\delta)^2}4&\le p^{2k} && \Longleftrightarrow m\le A_2 \nonumber. 
\end{align}
The only part of the verification which is not completely trivial is coprimality  of the entries of the triple in~\eqref{ethirdcase}. To see this, just note that the difference of the third and the first entry is  $\delta p^k$. This already proves coprimarity in the case of even~$\Delta$, when ${\delta=1}$, because none of the entries is divisible by~$p$. 
And if~$\Delta$ is odd, in which case ${\delta=2}$, we simply note that the middle entry is odd, because 
${m\equiv3\mod 4}$ (and this is because both~$\Delta$ and $p^{2k}$ are ${1\mod 4}$).

Thus, we have proved that ${\min \bigl\{p^{2k}, (m+(p^k-\delta)^2)/4\bigr\}}$
is suitable for~$\Delta$ when ${m\ge A_0}$. It remains to note that ${A_0\le (4/9)p^{2k}}$. Indeed, for ${x\ge 1}$ the function ${x\mapsto(x^2+2\delta x-3\delta^2)/3x^2}$ admits global maximum at ${x=3\delta}$, and this maximum is equal to $4/9$. 
\end{proof}

\begin{proposition}
\label{psuitthree}
Let~$a$ be a  prime number or ${a=4}$. If ${|\Delta|\ge 1000}$ then~$\Delta$ admits a suitable integer distinct from~$1$ and~$a$. 
\end{proposition}

\begin{proof}
There can exist at most~$2$ integers~$b$ satisfying 
$$
-a<b\le a, \qquad b^2\equiv \Delta\mod 4a.
$$ 
Therefore if  the only suitable integers for~$\Delta$ are~$1$ and~$a$ then the set~$T_\Delta$  consists of at most~$3$ elements. Hence ${h(\Delta)\le 3}$. This contradicts the assumption ${|\Delta|\ge 1000}$, because  the largest discriminant with class number~$3$ is $-907$, see Subsection~\ref{ssthree}. 
\end{proof}

\section{Roots of trinomials and the principal inequality}
\label{sroots}

In this section we establish some elementary metrical properties of roots of trinomials, both in the complex and $p$-adic  setting. The complex result, applied to singular moduli, will yield that non-dominant singular moduli of trinomial discriminant are very close to each other in absolute value. We call this  the ``principal inequality''; it will indeed be of crucial importance for the rest of the article. 

Everything is based on the following property: if ${x_0,x_1,x_2}$ are roots of a trinomial  ${t^m+At^n+B}$ then 
\begin{equation}
\label{edetzero}
\begin{vmatrix}
x_0^m&x_0^n&1\\
x_1^m&x_1^n&1\\
x_2^m&x_2^n&1
\end{vmatrix}=0. 
\end{equation}

\subsection{The complex case}

We start from the following observation. 

\begin{proposition}
Let ${x_0,x_1,x_2\in \C}$ be roots of a trinomial ${t^m+At^n+B\in \C[t]}$. Assume that  ${|x_0|\ge |x_1|\ge |x_2|}$.  Then 
\begin{equation}
\label{einter}
|1-(x_2/x_1)^n|\le 2|x_1/x_0|^{m-n} +2|x_1/x_0|^{m}
\end{equation}
and
\begin{equation}
\label{ewithout}
1-|x_2/x_1|\le 2|x_1/x_0| +2|x_1/x_0|^3. 
\end{equation}
\end{proposition}

While~\eqref{ewithout} is weaker than~\eqref{einter}, it has the advantage that~$m$ and~$n$ are not involved. Hence it may be used to check whether given numbers are roots of some trinomial. This will be used in Section~\ref{slittle}.

\begin{proof}
We may assume that ${|x_0|> |x_1|}$ and ${x_1\ne x_2}$, otherwise both results are trivial;  in particular, ${m\ge 3}$. 
Expanding the determinant in~\eqref{edetzero}, we obtain
\begin{align*}
|x_0|^m|x_1^n-x_2^n|&\le |x_0|^n\bigl(|x_1|^m+|x_2|^m\bigr)+ |x_1|^m|x_2|^n+|x_1|^n|x_2|^m\\
&\le2 |x_0|^n|x_1|^m+2|x_1|^{m+n}. 
\end{align*}
Now~\eqref{einter} follows dividing by ${|x_0|^m|x_1|^n}$, and~\eqref{ewithout} is immediate from~\eqref{einter}. 
\end{proof}

Recall that we call~$\Delta$ a trinomial discriminant of signature $(m,n)$ if ${h(\Delta)\ge 3}$ and the singular moduli of discriminant~$\Delta$ are roots of a trinomial of the form ${t^m+At^n+B}$ with rational coefficients. 


\begin{corollary}[The ``principal inequality'']
\label{cprinceq}
Let~$x_1$ and~$x_2$ be non-dominant singular moduli of trinomial discriminant~$\Delta$ of signature $(m,n)$. Assume that  ${|\Delta|\ge1000}$ and ${|x_1|\ge |x_2|}$. 
Then 
\begin{equation}
\label{eprinceq}
|1-(x_2/x_1)^n|\le e^{(m-n)(-\pi|\Delta|^{1/2}+\log|x_1|+10^{-20})+\log2}.
\end{equation}
In particular, we have the inequalities
\begin{align}
\label{eprineqmn}
1-|x_2/x_1|&\le e^{(m-n)(-\pi|\Delta|^{1/2}+\log|x_1|+10^{-20})+\log2},\\
\label{eprineq}
1-|x_2/x_1|&\le e^{-\pi|\Delta|^{1/2}+\log|x_1|+0.7},\\
\label{eprineqhalf}
1-|x_2/x_1|&\le e^{-\pi|\Delta|^{1/2}/2+0.7}.
\end{align}
\end{corollary}

\begin{proof}
Let~$x_0$ be the dominant singular modulus of discriminant~$\Delta$; in particular, ${|x_0|>|x_1|\ge|x_2|}$.  
Since ${|\Delta|\ge 1000}$, we have
\begin{align*}
&\left|\frac{x_1}{x_0}\right| < \frac{e^{\pi|\Delta|^{1/2}/2}+2079}{e^{\pi|\Delta|^{1/2}}-2079}
<10^{-21},\qquad
|x_0|> e^{\pi|\Delta|^{1/2}}-2079 > e^{\pi|\Delta|^{1/2}-10^{-30}}.
\end{align*}
Substituting this to~\eqref{einter}, we obtain
\begin{align*}
\left|1-\left(\frac{x_2}{x_1}\right)^n\right|&\le 2\left|\frac{x_1}{x_0}\right|^{m-n}+ 2\cdot(10^{-21})^n \left|\frac{x_1}{x_0}\right|^{m-n}\\
&\le 2(1+10^{-21})e^{(m-n)(-\pi|\Delta|^{1/2} +\log|x_1|+10^{-30})}\\
&\le e^{(m-n)(-\pi|\Delta|^{1/2}+\log|x_1|+10^{-20})+\log2}, 
\end{align*}
which proves~\eqref{eprinceq}.

Inequality~\eqref{eprineqmn} follows from~\eqref{eprinceq} and   both~\eqref{eprineq} and~\eqref{eprineqhalf} follow from~\eqref{eprineqmn}, because ${m-n\ge 1}$ and 
$$
\log|x_1|< \log(e^{\pi|\Delta|^{1/2}/2}+2079)<\pi|\Delta|^{1/2}/2+10^{-17} 
$$
when ${|\Delta|\ge 1000}$. 
\end{proof}

\subsection{The non-archimedean case}
In this subsection~$K$ is a 
field of characteristic~$0$ complete with respect to a non-archimedean absolute value ${|\cdot|}$. 
The results of this subsection will be used only in Subsection~\ref{ssthree} and Section~\ref{ssign}. 

\begin{proposition}
\label{pnapi}
\begin{enumerate}

\item
\label{ionlytwo}
The roots of  a trinomial ${f(t)=t^m+At^n+B\in K[t]}$ may have at most~$2$ distinct 
absolute values. In other words,  the set 
${\{|x|: x\in K, f(x)=0\}}$
consists of at most~$2$  elements. 

\item
\label{ijusttwo}
Assume now this set has exactly~$2$ distinct elements~$a$ and~$b$, with ${a>b}$. Then for any roots ${x_1, x_2}$ with ${|x_1|=|x_2|=b}$ we have 
\begin{equation}
\label{etwosmall}
|1-(x_2/x_1)^n|\le (b/a)^{m-n},  
\end{equation}
and for any roots ${x_1, x_2}$ with ${|x_1|=|x_2|=a}$ we have 
\begin{equation}
\label{etwobig}
|1-(x_2/x_1)^{m-n}|\le (b/a)^{n},  
\end{equation}
\end{enumerate}

\end{proposition}

\begin{proof}
If $x_0,x_1,x_2$ are roots with ${|x_0|>|x_1|>|x_2|}$ then the  determinant in~\eqref{edetzero} has the term $x_0^mx_1^n$ which has a strictly larger  absolute value 
than the other~$5$ terms. Hence the determinant cannot vanish. 
This proves item~\ref{ionlytwo}.

Now let $x_0,x_1,x_2$ be roots with ${|x_0|=a>|x_1|=|x_2|=b}$. Again expanding the determinant, we obtain
${a^m|x_1^n-x_2^n|\le a^nb^m}$, 
which proves~\eqref{etwosmall}.   

Finally, let $x_0,x_1,x_2$ be roots with ${|x_0|=b<|x_1|=|x_2|=a}$. Then the trinomial ${t^m+B^{-1}At^{m-n}+B^{-1}}$ has roots $x_0^{-1},x_1^{-1},x_2^{-1}$ satisfying
$$
|x_0^{-1}|=b^{-1}>|x_1^{-1}|=|x_2^{-1}|=a^{-1}. 
$$
Applying~\eqref{etwosmall} in this this set-up, we obtain 
${|1-(x_1^{-1}/x_2^{-1})^{m-n}|\le (a^{-1}/b^{-1})^{n}}$, 
which is~\eqref{etwobig}. 
\end{proof}

It turns out that   a trinomial having roots of~$2$ distinct absolute values  must have, in the algebraic closure~$\bar K$, exactly~$n$ ``small'' roots and ${m-n}$ ``big'' roots. 

\begin{proposition}
\label{pnumroots}
In the set-up  of item~\ref{ijusttwo} of Proposition~\ref{pnapi}, the trinomial
 $f(t)$ has exactly~$n$ roots  ${x\in \bar K}$ with ${|x|=b}$ and exactly ${m-n}$ roots ${x\in \bar K}$ with ${|x|=a}$ (both counted with multiplicities).
\end{proposition}
\begin{proof}
Denote ${x_1, \ldots, x_m}$ the roots of $f(t)$ in~$\bar K$ counted with multiplicities. Let~$k$ be the number of roots of  absolute value~$a$. The coefficient of $t^{m-k}$ in $f(t)$ is given by 
$$
(-1)^k\sum_{1\le i_1<\ldots<i_k\le m}x_{i_1}\cdots x_{i_k}. 
$$
In this sum exactly one term is of absolute value $a^k$, while the other terms are of strictly smaller absolute value. Hence the coefficient of $t^{m-k}$ does not vanish, which implies that ${m-k=n}$. 
\end{proof}

\begin{remark}
As the anonymous referees suggested, item~\ref{ionlytwo} of Proposition~\ref{pnapi}, and Proposition~\ref{pnumroots} can be proved using the ``Newton polygons'', as in Section~6.3 of~\cite{Ca86}. While our proof of Proposition~\ref{pnumroots} does not use Newton polygons, the present simple argument was inspired by the referees' comments.   Our initial statement of Proposition~\ref{pnumroots} was weaker, and the proof was long and ugly.
\end{remark}


\section{Suitable integers for trinomial discriminants}

\label{ssuittrin}

In this section~$\Delta$ denotes a trinomial discriminant unless the contrary is stated explicitly.  
The following property is crucial.

\begin{proposition}
\label{pflorian}
Let~$\Delta$ be a trinomial discriminant  admitting at least~$2$  distinct suitable integers other than~$1$. 
Let ${a>1}$ be suitable for~$\Delta$. Then we have ${a> 3|\Delta|^{1/2}/\log|\Delta|}$ if ${|\Delta|\ge10^5}$, and  ${a> 4|\Delta|^{1/2}/\log|\Delta|}$ if ${|\Delta|\ge 10^{10}}$. 
\end{proposition}

(It follows from the proof that, assuming~$|\Delta|$ large enough,~$4$ can be replaced by any ${c<2\pi}$.)


Here are some immediate consequences. 

\begin{corollary}
\label{cdeltapone}
Let~$\Delta$ be trinomial and~$p$  a prime number such that ${(\Delta/p)=1}$.   Then  ${p> 3|\Delta|^{1/2}/\log|\Delta|}$ if  ${|\Delta|\ge 10^5}$ and  ${p> 4|\Delta|^{1/2}/\log|\Delta|}$ if ${|\Delta|\ge 10^{10}}$. 
\end{corollary}

\begin{proof}
Assume that ${|\Delta|\ge 10^5}$ and ${p\le 3|\Delta|^{1/2}/\log|\Delta|}$. Then ${p\le |\Delta|^{1/2}/2}$, which implies that~$p$ is suitable for~$\Delta$ by item~\ref{ikrone} of Proposition~\ref{psuit}. Proposition~\ref{psuitthree} implies now that~$\Delta$ admits a suitable integer other than~$1$ and~$p$. Hence ${p> 3|\Delta|^{1/2}/\log|\Delta|}$ by Proposition~\ref{pflorian}, a contradiction. The case ${|\Delta|\ge 10^{10}}$ is treated similarly. 
\end{proof}

\begin{corollary}
\label{ccoprime}
Let~$\Delta$ be trinomial, ${|\Delta|\ge 10^5}$, and  ${a>1}$  suitable for~$\Delta$. Assume that ${\gcd(a,\Delta)=1}$. Then~$a$ is a prime number, and  ${a> 3|\Delta|^{1/2}/\log|\Delta|}$. Moreover, if ${|\Delta|\ge 10^{10}}$ then ${a> 4|\Delta|^{1/2}/\log|\Delta|}$. 
\end{corollary}

\begin{proof}
If~$a$ is composite then it has a prime divisor~$p$ satisfying ${p\le a^{1/2}}$. This~$p$ is also suitable for~$\Delta$ by item~\ref{idiv} of Proposition~\ref{psuit}, and Proposition~\ref{pflorian} implies that ${p> 3|\Delta|^{1/2}/\log|\Delta|}$. Hence 
$$
|\Delta|\ge 3a^2\ge3p^4\ge 243|\Delta|^2/(\log|\Delta)^4,
$$
or ${(\log|\Delta|)^4\ge 243|\Delta|}$, which is clearly impossible when ${|\Delta|\ge 10^5}$. 
Thus,~$a$ is prime, and we complete the proof using Corollary~\ref{cdeltapone}. 
\end{proof}



Before proving Proposition~\ref{pflorian}, we obtain the following preliminary statement.

\begin{proposition}
\label{pnotwo}
Let~$\Delta$ be a trinomial discriminant, ${|\Delta|\ge 10^5}$. Then it admits at most one suitable~$a$ satisfying ${1< a\le 3.4|\Delta|^{1/2}/\log|\Delta|}$. If ${|\Delta|\ge 10^{10}}$ then $3.4$ can be replaced by $4.5$. 
\end{proposition}

\begin{proof}
Assume that~$\Delta$ admits suitable~$a_1$ and~$a_2$ satisfying 
$$
1<a_1<a_2\le \kappa\frac{|\Delta|^{1/2}}{\log|\Delta|},
$$
with~$\kappa$ to be specified later. Let ${(a_1,b_1,c_1),(a_2,b_2,c_2)\in T_\Delta}$ be triples where our $a_1,a_2$  occur, and $x_1,x_2$ the corresponding singular moduli. 
We have
\begin{align*}
|x_1|-|x_2|&\ge e^{\pi|\Delta|^{1/2}/a_1}-e^{\pi|\Delta|^{1/2}/a_2}-4158\\
&=e^{\pi|\Delta|^{1/2}/a_2}\bigl(e^{\pi|\Delta|^{1/2}(1/a_1-1/a_2)}-1\bigr)-4158\\
&\ge e^{\pi\log|\Delta|/\kappa}\cdot \pi|\Delta|^{1/2}\left(\frac1{a_1}-\frac1{a_2}\right)-4158\\
&\ge \frac{\pi|\Delta|^{\pi/\kappa+1/2}}{a_1a_2}-4158\\
&\ge \frac{\pi}{\kappa^2}|\Delta|^{\pi/\kappa-1/2}(\log|\Delta|)^2-4158.\\
\end{align*}
A calculation shows that 
$$
\frac{\pi}{\kappa^2}|\Delta|^{\pi/\kappa-1/2}(\log|\Delta|)^2-4158\ge 500
$$
when ${\kappa =3.4}$ and ${|\Delta|\ge 10^5}$, or when ${\kappa =4.5}$ and ${|\Delta|\ge 10^{10}}$. (In fact, in the latter case~$500$ can be replaced by $3600$.) 
Hence  ${1-|x_2/x_1|\ge 500|x_1|^{-1}}$. Comparing this with the ``principal inequality''~\eqref{eprineqhalf}, we obtain 
$$
|x_1|\ge 500e^{\pi|\Delta|^{1/2}/2-0.7}, 
$$
which is impossible because ${|x_1|\le e^{\pi|\Delta|^{1/2}/2}+2079}$. 
\end{proof}

\begin{proof}[Proof of Proposition~\ref{pflorian}]
Let~$1$ and ${a_1>1}$ be the smallest suitable integers for~$\Delta$. By the assumption,~$\Delta$ admits a suitable integer ${a_2>a_1}$. 
Proposition~\ref{pnotwo} implies that ${a_2\ge 3.4|\Delta|^{1/2}/\log|\Delta|}$, where~$3.4$ can be replaced by $4.5$ if ${|\Delta|\ge 10^{10}}$.   

Now assume that ${a_1\le 3|\Delta|^{1/2}/\log|\Delta|}$. We will see that this leads to a contradiction.  We again let~$x_1$ and~$x_2$ be singular moduli for~$a_1$ and~$a_2$. Then 
\begin{equation}
\label{exoverx}
\frac{|x_2|}{|x_1|}\le \frac{\pi|\Delta|^{\pi/3.4}+2079}{\pi|\Delta|^{\pi/3}-2079} \le \frac{|\Delta|^{-0.12}+2079\pi^{-1}|\Delta|^{-\pi/3}}{1-2079\pi^{-1}|\Delta|^{-\pi/3}}.
\end{equation}
When ${|\Delta|\ge 10^5}$, the right-hand side of~\eqref{exoverx} does not exceed $0.3$. Then 
$$
1-|x_2/x_1|\ge 0.7, 
$$
which clearly contradicts~\eqref{eprineqhalf}. 

Now assume that ${|\Delta|\ge 10^{10}}$ and ${a_1\le 4|\Delta|^{1/2}/\log|\Delta|}$.  Then instead of~\eqref{exoverx} we have 
$$
\frac{|x_2|}{|x_1|}\le \frac{\pi|\Delta|^{\pi/4.5}+2079}{\pi|\Delta|^{\pi/4}-2079}= \frac{|\Delta|^{-\pi/36}+2079\pi^{-1}|\Delta|^{-\pi/4}}{1-2079\pi^{-1}|\Delta|^{-\pi/4}}.
$$
The right-hand side is again bounded by $0.3$, and we complete the proof in the same way. 
\end{proof}

Another important property of suitable integers is that they are of the same order of magnitude. 

\begin{proposition}
\label{ptight}
Let $a_1,a_2\ne 1$ be  suitable integers for a trinomial discriminant~$\Delta$. Then ${a_2<5a_1}$.
\end{proposition}

\begin{proof}
We assume that ${a_2\ge 5a_1}$ and will obtain a contradiction arguing as in the proof of Proposition~\ref{pnotwo}. Let~$x_1$ and~$x_2$ be singular moduli corresponding to~$a_1$ and~$a_2$. Then   
$$
|x_1|-|x_2|\ge e^{\pi|\Delta|^{1/2}/a_2} \cdot \pi|\Delta|^{1/2}\left(\frac1{a_1}-\frac1{a_2}\right)-4158.
$$
Using  ${5a_1\le a_2\le |\Delta/3|^{1/2}}$, we obtain
\begin{align*}
|x_1|-|x_2| &\ge  e^{\pi|\Delta|^{1/2}/a_2}  \cdot\pi|\Delta|^{1/2}\cdot\frac4{a_2}-4158\\ 
&\ge e^{\pi\sqrt3}\cdot 4\pi\sqrt3-4158\\
&>800. 
\end{align*}
Hence ${1-|x_2/x_1|>800/|x_1|}$, which leads to a contradiction exactly as in the proof of Proposition~\ref{pnotwo}. 
\end{proof}

\section{Small discriminants}
\label{slittle}

Recall that we call a discriminant~$\Delta$ trinomial if ${h(\Delta)\ge 3}$ and singular moduli of this discriminant are roots of a trinomial with rational coefficients. 

In this section we show that trinomial discriminants must be odd and not too small.

\begin{theorem}
\label{thsmall}
Let~$\Delta$ be trinomial discriminant. Then~$\Delta$ is odd and satisfies ${|\Delta|> 10^{11}}$. 
\end{theorem}

This theorem is an immediate consequence of Theorem~\ref{thhbiggerthree}, Proposition~\ref{podd} and Theorem~\ref{thh=three} proved below. 

Most of the arguments of this section are computer-assisted. We use packages \textsf{PARI}~\cite{pari} and \textsf{SAGE}~\cite{sagemath}. The reader may consult \url{https://github.com/yuribilu/trinomials} to view our \textsf{PARI} scripts. All computation were performed on a personal computer with {2.70}~GHz processor and {16.0}~GB RAM.


\subsection{Small discriminants with class number larger than~$3$}

In this subsection we prove the following. 

\begin{theorem}
\label{thhbiggerthree}
There are no trinomial discriminants~$\Delta$ with ${|\Delta|\le 10^{11}}$ and ${h(\Delta)>3}$. 
\end{theorem}

First of all, we show that sufficiently large trinomial discriminants must be odd.

\begin{proposition}
\label{podd}
Let~$\Delta$ be a trinomial discriminant with  ${|\Delta|\ge 10^5}$. Then~$\Delta$ is odd. 
\end{proposition}

\begin{proof}
If~${\Delta}$ is even but ${\Delta\not\equiv 4\mod32}$ then, according to item~\ref{ieven} of Proposition~\ref{psuit}, there is ${a\in \{2,4\}}$  suitable for~$\Delta$. Proposition~\ref{psuitthree} implies that~$\Delta$ admits  a suitable integer other than~$1$ and~$a$, and Proposition~\ref{pflorian} implies that ${4\ge a\ge 3|\Delta|^{1/2}/(\log|\Delta|)}$, which is impossible when ${|\Delta|\ge 10^5}$. 

If  ${\Delta\equiv 4\mod32}$ then  item~\ref{ihensel} of Proposition~\ref{psuit} tells us that~$8$ and~$16$ are suitable for~$\Delta$. Hence ${8\ge 3|\Delta|^{1/2}/(\log|\Delta|)}$ by Proposition~\ref{pflorian}, which is again  impossible when ${|\Delta|\ge 10^5}$. 
\end{proof}

Next, we dispose of the discriminants in the range ${|\Delta|\le 10^5}$. 

\begin{proposition}
\label{pletentofive}
There are no trinomial discriminants~$\Delta$ with ${|\Delta|\le 10^5}$ and ${h(\Delta)>3}$. 
\end{proposition}

\begin{proof}
The proof is by a \textsf{PARI} script. For every such~$\Delta$ our script finds singular moduli ${x_0,x_1,x_2}$ of discriminant~$\Delta$ such that inequality~\eqref{ewithout} does not hold. More precisely, for every~$\Delta$  in the range ${|\Delta|\le 10^5}$, except ${\Delta=-1467}$, our script finds ${x_0,x_1,x_2}$ satisfying 
$$
1-|x_2/x_1|>2|x_1/x_0|+2|x_1/x_0|^3+0.15.
$$
For the exceptional ${\Delta=-1467}$ one has the same inequality, but with 0.001 instead of $0.15$. 

The total running time was less than 6 minutes. 
\end{proof}

Unfortunately, this method fails for discriminants with class number~$3$. Each of those admits one real singular modulus~$x_0$ (the dominant one) and two complex conjugate singular moduli ${x_1, x_2=\bar x_1}$; for them inequality~\eqref{ewithout} is trivially true. In Subsection~\ref{ssthree} we use a totally different method to show that discriminants with ${h=3}$ cannot be trinomial. 

To dismiss larger discriminants, we show that they admit a small prime~$p$ with  ${(\Delta/p)=1}$.

\begin{proposition}
\begin{enumerate}
\item
Every odd discriminant~$\Delta$ with ${|\Delta|\le 10^{11}}$ admits a prime ${p\le 163}$ such that ${(\Delta/p)=1}$. 

\item
Every odd discriminant~$\Delta$ with ${|\Delta|\le 10^{6}}$ admits a prime ${p\le 79}$ such that ${(\Delta/p)=1}$. 
\end{enumerate}
\end{proposition}

\begin{proof}
We use again a \textsf{PARI} script. It works in 3 steps. In what follows~$X$ is a (large) positive number and~$p_0$ is the largest prime number such that 
$$
N_0=8\prod_{3\le p\le p_0}p <X,  
$$
the product being over primes~$p$ in the indicated range. Let also $p_1,p_2$ be the first two primes larger than~$p_0$. We have ${p_0 =29}$, ${p_1=31}$,   ${p_2=37}$ for ${X=10^{11}}$ and    
${p_0 =13}$, ${p_1=17}$,   ${p_2=19}$ for ${X=10^{6}}$.

\paragraph{Building the list of residues}
In the first step we use successively the Chinese Remainder Theorem to generate the list of residues ${n\mod N_0}$ such that ${n\equiv 5\bmod 8}$ and ${(n/p)\ne 1}$ for every odd prime ${p\le p_0}$. There are 
$$
\prod_{3\le p\le p_0} \frac{p+1}{2}
$$
such residues altogether, which gives $16329600$ residues for ${X=10^{11}}$ and    $1008$ residues 
 for ${X=10^{6}}$. 

\paragraph{Building the list of discriminants}
For every residue class ${n\mod N_0}$ from the previous list, and  every residue class ${m\mod p_1}$ with ${(m/p_1)\ne 1}$ we find the smallest \textit{negative} number~$\Delta$ belonging to both, and we include this~$\Delta$ in the list only if ${|\Delta|\le X}$. We obtain the full list of odd discriminants $\Delta$ with the properties 
${|\Delta|\le X}$ and  ${(\Delta/p)\ne 1}$ for all ${p\le p_1}$ (including ${p=2}$). 
We end up with  32567861 discriminants for ${X=10^{11}}$ and with 4450 discriminants for ${X=10^{6}}$.

\paragraph{Sieving}
Now we sieve our list modulo every prime ${p\ge p_2}$, by deleting from the list the discriminants~$\Delta$ with ${(\Delta/p)=1}$. The list was emptied after ${p=163}$ for ${X=10^{11}}$ and after ${p=79}$  for ${X=10^{6}}$.

\bigskip

The bottleneck steps are building the list of discriminants and sieving modulo~$p_2$: they require most of processor time and memory. The total running time  was less than 5 minutes for ${X=10^{11}}$ and less than 0.1 second for ${X=10^6}$.
\end{proof}

Now we are ready to prove Theorem~\ref{thhbiggerthree}.

\begin{proof}[Proof of Theorem~\ref{thhbiggerthree}]
The range ${|\Delta|\le 10^5}$ is Proposition~\ref{pletentofive}. Now assume that~$\Delta$ is trinomial in the range ${10^5\le |\Delta|\le 10^{11}}$. Then~$\Delta$ is odd by Proposition~\ref{podd}. If ${10^6\le |\Delta|\le 10^{11}}$ then ${(\Delta/p)=1}$ for some prime ${p\le 163}$. Corollary~\ref{cdeltapone} implies that ${163>3|\Delta|^{1/2}/\log|\Delta|}$, which is impossible  when ${|\Delta|\ge10^6}$. 

Similarly, if ${10^5\le |\Delta|\le 10^6}$ then ${(\Delta/p)=1}$ for some prime ${p\le 79}$. Hence ${79>3|\Delta|^{1/2}/\log|\Delta|}$, which is impossible  when ${|\Delta|\ge10^5}$. 
\end{proof}


\subsection{Discriminants with class number~$3$}
\label{ssthree}

In this subsection we prove the following theorem.

\begin{theorem}
\label{thh=three}
There are no trinomial discriminants with class number~$3$. 
\end{theorem}

There are 25 discriminants with class number~$3$: the full list of them, found by \textsf{SAGE} command \textsf{cm\_orders}, is the top row of Table~\ref{tapenlen} below. As we have seen, they cannot be dismissed using the method of Proposition~\ref{pletentofive}. 
Instead, we use a version of the argument from~\cite{LR19}. 





We start by some general discussion. 
Assume that we are in the following situation: 
${F(t)=t^3+at^2+bt+c\in \Z[t]}$
is a $\Q$-irreducible polynomial with splitting field~$L$ of degree~$6$, and~$p$ is a prime number such that 
\begin{equation}
\label{epcnb}
\text{${p\mid c}$, but ${p\nmid b}$.}
\end{equation}
Assume further that 
\begin{equation}
\label{eunraminert}
\text{$p$ is unramified in~$L$ and inert in the quadratic subfield of~$L$.}
\end{equation}
(The latter assumption can be suppressed, but it holds in all cases that interest us, and many arguments below simplify when it is imposed.)

Let ${x_0,x_1,x_2\in L}$ be the roots of~$F$. Since ${p\mid c}$, every~$x_i$ must be divisible by a prime ideal above~$p$. Since ${p\nmid b}$, these ideals must be distinct. Hence~$p$ splits in~$L$ in at least~$3$ distinct primes. Since~$p$ is inert in the quadratic subfield, it splits in exactly~$3$ primes: ${p=\gerp_0\gerp_1\gerp_2}$, with ${\gerp_i\mid x_i}$. Since~$p$ is unramified, we have 
${\nu_{\gerp_i}(x_i)=\nu_p(c)}$ for  ${i=0,1,2}$.

Now assume that $F(t)$ divides a trinomial ${t^m+At^n+B\in \Q[t]}$. Then, setting ${\theta=x_2/x_1}$, Proposition~\ref{pnapi} implies that 
\begin{equation}
\label{eiftrinomial}
\nu_{\gerp_0}(1-\theta^{m-n}) \ge n\nu_{\gerp_0}(x_0)= n\nu_p(c). 
\end{equation}
To make use of this, we need the following classical fact. Some versions of it were known already to Lucas~\cite{Lu78} or even earlier, but we prefer to include the proof for the reader's convenience.

\begin{proposition}
\label{plu}
Let~$p$ be a prime number,~$L$ a number field, ${\gerp\mid p}$ a prime of~$L$ and ${\theta\in L}$ a $\gerp$-adic unit, but not a root of unity. Assume that~$\gerp$ is unramified over~$p$. Let~$r$ be the  order of ${\theta \bmod \gerp}$ (the smallest positive integer such that ${\theta^r\equiv 1\mod \gerp}$). Then for every positive integer~$k$ such that ${\theta^k\equiv 1\mod \gerp}$ we have ${r\mid k}$. Furthermore, for ${p>2}$ we have 
\begin{equation}
\label{elu}
\nu_p(k)=
\nu_\gerp(1-\theta^k)-\nu_\gerp(1-\theta^r), 
\end{equation}
and for ${p=2}$ we have 
\begin{equation}
\label{elutwo}
\nu_2(k)=
\begin{cases}
\nu_\gerp(1-\theta^k)-\nu_\gerp(1-\theta^r), & 2\nmid k,\\
\nu_\gerp(1-\theta^k)-\nu_\gerp(1-\theta^{2r})+1, & 2\mid k. 
\end{cases}
\end{equation}
\end{proposition}

\begin{proof}
We only have to prove~\eqref{elu} and~\eqref{elutwo}, because ${r\mid k}$ is obvious. Note also that~$r$ divides ${\norm\gerp-1}$, the order of the multiplicative group $\bmod\gerp$; in particular, ${p\nmid r}$.

Assume first that ${p\nmid k}$. Then 
$$
1-\theta^k=(1-\theta^r)(1+\theta^r+\theta^{2r}+\cdots +\theta^{k-r}). 
$$
Since ${\theta^r\equiv 1\mod \gerp}$, the second factor is congruent to~$k/r$ modulo~$\gerp$. In particular, it is not divisible by~$\gerp$. Hence ${\nu_\gerp(1-\theta^k)=\nu_\gerp(1-\theta^r)}$, and both~\eqref{elu},~\eqref{elutwo} are true in this case.

Now assume that ${p\mid k}$ and ${p>2}$.  Write   ${\theta^{k/p}=1+\beta}$. Since ${r\mid k/p}$, we have ${\nu_\gerp(\beta)\ge 1}$. Hence
$$
\theta^k=1+p\beta+\beta^p +(\text{terms of $\gerp$-adic valuation  $>\nu_\gerp(p\beta)$}). 
$$
Since~$\gerp$ is unramified and ${p>2}$, we have ${\nu_\gerp(\beta^p)>\nu_\gerp(p\beta)}$ as well. Hence 
$$
\nu_\gerp(1-\theta^k)=\nu_\gerp(p\beta) =1+\nu_\gerp(1-\theta^{k/p}),
$$
and~\eqref{elu} follows by induction in  $\nu_p(k)$. 

Now let ${p=2}$. When ${2\,\|\,k}$ we can prove that ${\nu_\gerp(1-\theta^k)=\nu_\gerp(1-\theta^{2r})}$ in the same way as we proved ${\nu_\gerp(1-\theta^k)=\nu_\gerp(1-\theta^r)}$ when ${p\nmid k}$. Hence~\eqref{elutwo} is true is this case. 
Now assume that ${4\mid k}$ and write 
$$
\theta^{k/4}=1+\alpha, \qquad \theta^{k/2}=1+\beta.
$$
Since ${r\mid k/4}$, we have ${\nu_\gerp(\alpha)\ge 1}$ and 
${\nu_\gerp(\beta)=\nu_\gerp(2\alpha+\alpha^2)\ge 2}$. 
Since~$\gerp$ is unramified, this implies that ${\nu_2(2\beta+\beta^2)=1+\nu_2(\beta)}$, or, in other words, 
$$
\nu_\gerp(1-\theta^k)=1+\nu_\gerp(1-\theta^{k/2}),
$$
and we complete the proof by induction as before. 
\end{proof}

Comparing Proposition~\ref{plu} and inequality~\eqref{eiftrinomial}, we obtain the following consequence.

\begin{corollary}
\label{copadic}
Let $F(t)$ be as above and let~$p$ satisfy~\eqref{epcnb} and~\eqref{eunraminert}. Define~$\theta$ and~$\gerp_0$ as above. Assume that $F(t)$ divides a trinomial ${t^m+At^n+B\in \Q[t]}$. Then
\begin{equation}
\label{elowerm-n}
m-n\ge r_0p^{\nu_p(c)n-\nu_0},
\end{equation}
where~$r_0$ is the order of ${\theta=x_2/x_1}$ modulo~$\gerp_0$, and 
$$
\nu_0=
\begin{cases}
\nu_{\gerp_0}(1-\theta^{r_0}), &p>2,\\
\nu_{\gerp_0}(1-\theta^{2r_0})-1, & p=2. 
\end{cases}
$$
\end{corollary}

\begin{proof}
We have ${\nu_{\gerp_0}(1-\theta^{m-n}) >0}$ by~\eqref{epcnb} and~\eqref{eiftrinomial}.  Proposition~\ref{plu} implies that ${r_0\mid m-n}$. Furthermore, for ${p>2}$ we use~\eqref{epcnb} and~\eqref{elu} to obtain 
$$
\nu_p(m-n) = \nu_{\gerp_0}(1-\theta^{m-n}) - \nu_{\gerp_0}(1-\theta^{r_0}) \ge n\nu_p(c)-\nu_0. 
$$
Thus, both~$r_0$ and ${p^{n\nu_p(c)-\nu_0}}$ divide ${m-n}$. Since ${p\nmid r_0}$, this proves~\eqref{elowerm-n} in the case ${p>2}$. 

In the case ${p=2}$ the same argument gives 
$$
\nu_2(m-n) \ge n\nu_2(c)-\max\{ \nu_{\gerp_0}(1-\theta^{r_0}), \nu_{\gerp_0}(1-\theta^{2r_0})-1\}.
$$
We have clearly ${\nu_{\gerp_0}(1-\theta^{2r_0})\ge \nu_{\gerp_0}(1-\theta^{r_0})+1}$, which implies that the maximum above is~$\nu_0$. 
\end{proof}

\begin{remark}
\label{rehowcalculate}
\begin{enumerate}
\item
The lower bound~\eqref{elowerm-n} is quite strong, but to profit from it in practical situations, we must be able to calculate~$r_0$ and~$\nu_0$. We do it as follows. Let $F_k(t)$ be the monic polynomial of degree~$3$ whose roots are $x_0^k,x_1^k,x_2^k$, and~$\discr_k$ its discriminant. Then 
${\nu_{\gerp_0}(1-\theta^k)=\nu_p(\discr_k)/2}$, because 
${\norm_{L/\Q}(x_1^k-x_2^k)=\discr_k}$ and  ${\norm_{L/\Q}\gerp_0=p^2}$. 
In particular,~$r_0$ is the smallest positive~$k$ for which ${p\mid \discr_k}$ and 
$$
\nu_0=
\begin{cases}
\nu_p(\discr_{r_0})/2,&p>2, \\
\nu_2(\discr_{2r_0})/2-1,&p=2.
\end{cases}
$$


\item
Polynomials ${F_k(t)=t^3+a_kt^2+b_kt+c_k}$ are very easy to calculate consecutively. Indeed, 
\begin{align*}
&F_0(t)=t^3-3t^2+3t-1, \qquad F_1(t)=F(t), \\ 
&F_2(t)=t^3+(-a^2+2b)t^2+(b^2-2ac)t-c^2, 
\end{align*}
and for general~$k$ we have ${c_k=-(-c)^k}$, 
$$
a_k=-x_0^k-x_1^k-x_2^k, \qquad b_k=(-c/x_0)^k +(-c/x_1)^k+(-c/x_2)^k,   
$$
which implies  the recurrence relations 
$$
a_{k+3}= -aa_{k+2}-ba_{k+1}-ca_k, \qquad b_{k+3}= bb_{k+2}-acb_{k+1}+c^2b_k. 
$$
\end{enumerate}
\end{remark}


Theorem~\ref{thh=three} is an easy consequence of the following statement. 

\begin{proposition}
\label{ppenlen}
Let~$\Delta$ be a discriminant with class number~$3$. Assume that~$\Delta$ is trinomial of signature $(m,n)$. Then 
${p^{3n}<\lambda n+\mu}$, where $p,\lambda,\mu$ can be found in Table~\ref{tapenlen}.

\begin{table}
\caption{Data for Proposition~\ref{ppenlen}}
\label{tapenlen}
{\tiny
\begin{align*}
&
\begin{array}{c|ccccccccccccc}
\Delta&
-23& -31& -44& -59& -76& -83& -92& -107& -108& -124& -139& -172& -211\\
\hline
p&17& 23& 29& 11& 53& 2& 53& 17& 17& 89& 23& 113& 29\\
r_0&1& 1& 1& 1& 18& 1& 1& 1& 1& 1& 1& 38& 1\\
\nu_0&1& 1& 1& 2& 1& 2& 1& 1& 1& 1& 1& 1& 1\\
\lambda&54 & 91 & 83 & 193& 6.4& 4.6& 150& 25  & 34  & 209 & 34  & 6   & 49\\
\mu    &6.4& 7.5& 5.8& 23 & 0.6& 0.8& 7.5& 2.4 & 2.1 & 8.8 & 2.4 & 0.3 & 2.4
\end{array}\\
&\begin{array}{c|cccccccccccc}
\Delta&
 -243& -268& -283& -307& -331& -379& -499& -547& -643& -652& -883& -907\\
\hline
p&23& 197& 53& 47& 59& 71& 83& 101& 113& 389& 113& 167\\
r_0&1& 11& 18& 1& 1& 1& 28& 1& 19& 2& 38& 56\\
\nu_0&1& 1& 1& 1& 2& 1& 1& 1& 1& 1& 1& 1\\
\lambda& 35&   41&  4  &   65& 6447&  138&  6.2&  152&  8.1&  480&  5.1&  4.7\\
\mu&    1.7&  1.5&  0.2&  3  &  215&  4.2&  0.3&  4.4&  0.5&  9.2&  0.2&  0.2
\end{array}
\end{align*}
}%
\end{table}

\end{proposition}

\begin{proof}
It is again by a \textsf{PARI} script. 
Let $H(t)$ be the Hilbert Class Polynomial of~$\Delta$ (the monic polynomial whose roots are the singular moduli of discriminant~$\Delta$). A verification with \textsf{PARI} shows that, for each of the 25 possible~$\Delta$, it has one real root (the dominant one) and two complex conjugate roots. 

Let~$d$ be the largest positive integer such that ${d^{-3}H(dt)\in \Z[t]}$. We want to apply Corollary~\ref{copadic} to the polynomial 
$$
F(t)= d^{-3}H(dt)=t^3+at^2+bt+c. 
$$  
We pick a prime number~$p$ such that ${p\mid c}$, but ${p\nmid b}$; our script shows the existence of at least one such~$p$ in all the 25 cases.  If there are several~$p$ with this property, we take the largest of them. The prime chosen for each~$\Delta$ can be seen in Table~\ref{tapenlen}. As we verified, each of our primes satisfies ${(\Delta/p)=-1}$, which means that it is  unramified in~$L$ (the splitting field of~$F$) and inert  in ${\Q(\sqrt\Delta)}$, the quadratic subfield of~$L$, so we are indeed in the set-up of Corollary~\ref{copadic}.

We calculate~$r_0$ and~$\nu_0$ as defined in Corollary~\ref{copadic}, using the method outlined in Remark~\ref{rehowcalculate}. Their values are in Table~\ref{tapenlen} as well.  Corollary~\ref{copadic} implies that~\eqref{elowerm-n} holds true. As our script verified, we always have ${\nu_p(c)=3}$. This implies the lower bound 
\begin{equation}
\label{elowerm-nspecific}
m-n\ge r_0p^{3n-\nu_0},
\end{equation}

Now let us bound ${m-n}$ in terms of~$n$ from above. Let~$x_0$ be the real root of~$F$ and ${x_1,x_2=\bar x_1}$ the complex conjugate roots. (Our definition of~$F$ implies that ${dx_0,dx_1,dx_2}$ are the singular moduli of discriminant~$\Delta$, of which one is real and the other two are complex conjugates.) Set ${\theta=x_2/x_1}$. Then we have inequality~\eqref{einter}:
$$
|1-\theta^n|\le 2|x_1/x_0|^{m-n}+2|x_1/x_0|^{m}.
$$
A  quick calculation with \textsf{PARI} shows that ${|x_1/x_0|< 0.001}$, which implies that 
\begin{equation}
\label{eoneminthnle}
|1-\theta^n|< 2.01|x_1/x_0|^{m-n}. 
\end{equation} 

On the other hand, we can estimate ${|1-\theta^n|}$ from below using the classical \textit{Liouville inequality}: if~$\alpha$
is a non-zero complex algebraic number of degree~$\delta$, then 
$$
|\alpha|\ge 
\begin{cases}
e^{-\delta\height(\alpha)}, & \alpha \in \R,\\
e^{-(\delta/2)\height(\alpha)}, & \alpha \notin \R. 
\end{cases}
$$
Here ${\height(\cdot)}$ is the usual absolute logarithmic height\footnote{There is no risk of confusing the height $\height(\cdot)$ and the class number $h(\cdot)$, not only because the former is roman and the latter is italic, but also  because the class number notation $h(\cdot)$ does not occur in this subsection,  and heights do not occur outside this subsection.}.

Our~$\theta$ is an algebraic number of degree~$6$, and not a root of unity, because among its $\Q$-conjugates there is ${x_0/x_1}$, of absolute value distinct from~$1$. Hence we may apply the Liouville inequality to ${\alpha=1-\theta^n}$. 

We have clearly ${\height(\alpha)\le n\height(\theta)+\log 2}$. To estimate $\height(\theta)$, note that  the $\Q$-conjugates of~$\theta$ are the~$6$ numbers ${x_i/x_j}$ with ${1\le i\ne j\le 3}$. Of them, only ${x_0/x_1}$ and ${x_0/x_2}$ are greater than~$1$ in absolute value. Also,~$\theta$ is a root of the polynomial 
$$
c^2\prod_{i\ne j}(t-x_i/x_j) \in \Z[t],
$$
where ${c=-x_0x_1x_2=-x_0|x_1|^2}$ is the free term of $F(t)$. Hence 
$$
\height(\theta)\le \frac16(2\log|x_0/x_1|+2\log|x_0x_1^2|)=\frac23\log|x_0|+\frac13\log|x_1|. 
$$
It follows that 
$$
|1-\theta^n|\ge e^{-(6/2)(n\height(\theta)+\log2)} \ge e^{-n( 2\log|x_0|+\log|x_1|) -3\log2}. 
$$
Together with~\eqref{eoneminthnle} this implies that
$$
m-n<  \frac{2\log|x_0|+\log|x_1|}{\log|x_0/x_1|} n+ \frac{3\log 2+\log 2.01}{\log|x_0/x_1|}. 
$$
Comparing this with the lower bound~\eqref{elowerm-nspecific}, we obtain ${p^{3n}<\lambda n+\mu}$ with 
$$
\lambda=  \frac{2\log|x_0|+\log|x_1|}{\log|x_0/x_1|}\frac{p^{\nu_0}}{r_0}, \qquad 
\mu = \frac{3\log 2+\log 2.01}{\log|x_0/x_1|}\frac{p^{\nu_0}}{r_0}. 
$$
Upper bounds for~$\lambda$ and~$\mu$ produced by our script can be found in Table~\ref{tapenlen}. 

The total running time   
was less than 2 seconds. 
\end{proof}

\begin{proof}[Proof of Theorem~\ref{thh=three}]
When ${\Delta\ne-83,-331}$,  inequality ${p^{3n}<\lambda n+\mu}$, where  $p,\lambda,\mu$ are as in Table~\ref{tapenlen},  cannot hold for ${n\ge 1}$.  Indeed, for these 23 values of~$\Delta$   we have ${p\ge 11}$, ${\lambda \le 480}$ and ${\mu \le 23}$. Hence we have ${11^{3n} <480n+23}$, which is impossible for ${n\ge 1}$. For the remaining two values of~$\Delta$ we have 
\begin{align*}
2^{3n}&<4.6n+0.8 &&(\Delta=-83), \\
59^{3n}&< 6447n+ 215&& (\Delta=-331).
\end{align*}
These inequalities are impossible for ${n\ge 1}$ as well. The theorem is proved. 
\end{proof}

\section{Structure of trinomial discriminants}
\label{sstruc}

In this section we prove Theorem~\ref{thfundintro}. For convenience, we reproduce the statement here. 

\begin{theorem}
\label{thfund}
A trinomial discriminant must be  of the form $-p$ or ${-pq}$, where $p,q$ are distinct odd prime numbers,
${p\not\equiv q\mod 4}$. In particular, a trinomial discriminant is fundamental. 
\end{theorem}

In this section~$\Delta$ denotes a trinomial discriminant; in particular,~$\Delta$ is odd and  ${|\Delta|\ge 10^{11}}$ by Theorem~\ref{thsmall}. 
The proof is split into many steps which correspond to Subsection~\ref{ssatmosttwo}--\ref{sskoneone} below. 

\subsection{$\Delta$ may have at most~$2$ prime divisors}
\label{ssatmosttwo}
Assume that~$\Delta$ has~$3$ distinct (odd) prime divisors ${p_1,p_2,p_3}$. Set  
${a_i=p_i^{\nu_{p_i}(\Delta)}}$ and ${a_i'=|\Delta/a_i|}$. 
We may assume that ${3\le a_1<a_2<a_3}$. We have clearly ${a_i'\ge 3a_i}$ for ${i=1,2}$. Item~\ref{icoprime} of Proposition~\ref{psuit} implies that both~$a_1$ and~$a_2$ are suitable for~$\Delta$. 
Using Proposition~\ref{pflorian} we obtain
${|\Delta|^{1/3}\ge a_1> 4|\Delta|^{1/2}/\log|\Delta|}$, 
which is impossible when ${|\Delta|\ge10^{11}}$.

. 

\subsection{$-\Delta$ is not a square}
\label{sssnotsquare}
Assume that ${\Delta=-m^2}$, with ${m\in \Z}$. 
Among the three primes ${5,13,17}$ there is one, call it~$q$, which does not divide~$\Delta$.  This~$q$ must be suitable, because ${(\Delta/q)=1}$. Corollary~\ref{cdeltapone} implies now that ${17\ge q\ge 4|\Delta|^{1/2}/\log|\Delta|}$, which is impossible when ${|\Delta|\ge 10^{11}}$. 

\subsection{If ${\Delta=-p^k}$  then ${\Delta=-p}$} 
Assume that ${\Delta=-p^k}$, where~$p$ is a prime number and~$k$ a positive integer. Since $-\Delta$ is not a square,~$k$ must be odd. Assume that ${k\ge 3}$. 
Let~$q$ be an odd prime divisor of ${p+4}$. Then ${(-p/q)=1}$, which implies that ${(\Delta/q)=1}$. In addition to this, ${|\Delta|\ge 10^{11}}$ implies that 
${|\Delta|\ge4(|\Delta|^{1/3}+4)^2\ge 4q^2}$. 
Hence~$q$ is suitable for~$\Delta$, and  Corollary~\ref{ccoprime} implies now that 
$$
|\Delta|^{1/3}+4\ge q \ge 4|\Delta|^{1/2}/\log|\Delta|,
$$
which is impossible when ${|\Delta|\ge 10^{11}}$. Thus, ${k=1}$ and ${\Delta=-p}$. 

\bigskip

\textit{We are left with the case when~$\Delta$ has exactly two odd prime divisors~$p_1$ and~$p_2$, with ${p_1<p_2}$. In the sequel we write 
$$
\Delta=-p_1^{k_1}p_2^{k_2}. 
$$
We want to show that ${k_1=k_2=1}$.}

\subsection{We have ${k_1\le 2}$}
\label{ssskonele}

Assume that ${k_1\ge 3}$.  Let us show first of all that we must have ${(k_1,k_2)=(3,1)}$.
Indeed, assume the contrary: ${k_1\ge 3}$ and ${k_1+k_2\ge 5}$. Writing ${\Delta=-p_1^2m}$, Proposition~\ref{pprimepower} implies that ${\min\{p_1^2, (m+(p_1-2)^2)/4\}}$ is suitable for~$\Delta$. However, since ${k_1+k_2\ge 5}$, we have ${m> 4p_1^2}$, which implies that the minimum is, actually, $p_1^2$. Thus, $p_1^2$ is suitable for~$\Delta$. Item~\ref{icoprime} of Proposition~\ref{psuit} implies that one of the three numbers 
$$
p_1^{k_1},\quad p_2^{k_2},\quad  (p_1^{k_1}+p_2^{k_2})/4
$$
is suitable as well. Since none of these numbers is equal to $p_1^2$, Proposition~\ref{pflorian} applies, and we obtain 
\begin{equation}
\label{etwofifthpone}
|\Delta|^{2/5}\ge p_1^2 \ge 4|\Delta|^{1/2}/\log|\Delta|,
\end{equation}
which is impossible when ${|\Delta|\ge 10^{11}}$.

Thus, we have ${(k_1,k_2)=(3,1)}$, that is, ${\Delta=-p_1^3p_2}$. In this case we have suitable integers 
$$
a_1=\min\{p_1^2, (p_1p_2+(p_1-2)^2)/4\}, \qquad a_2=\min\{p_1^3,p_2,(p_1^3+p_2)/4\}. 
$$
Let us show that ${a_1\ne a_2}$. If ${a_1=a_2}$ then we must have  
$$
a_1=(p_1p_2+(p_1-2)^2)/4<p_1^2, \qquad a_2=(p_1^3+p_2)/4. 
$$
It follows that 
${p_1^3<4a_2=4a_1< 4p_1^2}$, which implies that ${p_1=3}$. Furthermore, ${p_1p_2<4a_1< 4p_1^2}$, which implies that ${p_2\le 11}$. 
Hence 
$$
|\Delta|=p_1^3p_2\le 3^3\cdot11<10^{11},
$$
a contradiction. 

Thus, ${a_1\ne a_2}$, and Proposition~\ref{pflorian} applies. If ${p_2<p_1^2}$ then ${a_2=p_2}$, and we have
${|\Delta|^{2/5}\ge p_2 \ge 4|\Delta|^{1/2}/\log|\Delta|}$,
which is impossible.  
Now assume that 
\begin{equation}
\label{efuckref}
p_2>p_1^2, 
\end{equation} 
in which case 
$$
\frac{p_1p_2+(p_1-2)^2}4>\frac{p_1}4p_1^2.
$$ 
When ${p_1\ge 5}$ this implies that ${a_1=p_1^2}$. When ${p_1=3}$ we have ${p_2\ge 10^{11}3^{-3}}$ which again implies 
${a_1=p_1^2}$. Thus, ${a_1=p_1^2}$ in any case. From~\eqref{efuckref} we deduce that ${|\Delta|>p_1^5}$, and we end up with~\eqref{etwofifthpone}. 

This shows that ${(k_1,k_2)\ne (3,1)}$. Hence we proved that ${k_1\le 2}$.

\subsection{We have ${k_2=1}$}

Assume that ${k_2\ge 2}$. If ${k_1=1}$ then~$p_1$ is suitable for~$\Delta$ by item~\ref{icoprime} of Proposition~\ref{psuit}. Proposition~\ref{psuitthree} implies that there must be a suitable integer distinct from~$1$ and~$p_1$. Proposition~\ref{pflorian} implies now that
$$
|\Delta|^{1/3}> p_1\ge 4|\Delta|^{1/2}/\log|\Delta|,
$$
which is impossible when ${|\Delta|\ge10^{11}}$.

Thus, ${k_1=2}$. Hence ${k_2\ge3}$, because $-\Delta$ is not a square.

Item~\ref{icoprime} of Proposition~\ref{psuit} implies that~$p_1^2$ is suitable. We want to show that there is one more suitable integer, distinct from~$1$ and~$p_1^2$. If ${k_2\ge4}$ then an easy application of Proposition~\ref{pprimepower} implies that~$p_2^2$ is suitable, so in the sequel we will assume that ${k_2=3}$, that is,
${\Delta=-p_1^2p_2^3}$.
Note that we must have ${p_1p_2\ge 6000}$: otherwise ${|\Delta|<(p_1p_2)^3/3<10^{11}}$, a contradiction.

We consider two cases.

\subsubsection{The case ${p_2\ge1.1p_1^2}$}

Pick ${\ell\in \{3,5\}}$ to have ${\ell\ne p_1}$. Since ${p_1^2p_2>3p_1p_2>10^4}$, the numbers ${p_1^2p_2+1}$ and ${p_1^2p_2+\ell^2}$ cannot be both powers of~$2$. Hence there exists an odd prime~$q$ dividing one of them. This~$q$ satisfies 
$$
q\le \frac{p_1^2p_2+25}2< 0.502p_1^2p_2. 
$$
Using the assumption ${p_2\ge1.1p_1^2}$, we obtain
$$
|\Delta|= p_2^2\cdot p_1^2p_2\ge1.1(p_1^2p_2)^2>4q^2. 
$$
We have clearly
${(\Delta/q)=(p_1^2p_2/q)=1}$. Hence~$q$ is suitable by item~\ref{ikrone} of Proposition~\ref{psuit}. 

\subsubsection{The case ${p_2\le1.1p_1^2}$}

{\sloppy

Write ${\Delta=-p_2^2m}$, where 
${m= p_1^2p_2 >0.9p_2^2}$. Proposition~\ref{pprimepower} implies that ${a=\min\{p_2^2,(m+(p_2-2)^2)/4\}}$ is suitable for~$\Delta$. We have clearly ${p_2^2>p_1^2}$, and also
$$
\frac{m+(p_2-2)^2}4>\frac m4\ge \frac{p_1^2p_2}4>p_1^2. 
$$
Hence ${a>p_1^2}$.

}

\bigskip

We have showed that in any case~$\Delta$ admits a suitable integer distinct from~$1$ and~$p_1^2$. Hence Proposition~\ref{pflorian} applies, and we again have~\eqref{etwofifthpone}, which leads to a contradiction. Thus, we must have ${k_2=1}$. 


\subsection{We have ${k_1=1}$}
\label{sskoneone}
The only remaining possibilities  are ${\Delta=-p_1^2p_2}$ and ${\Delta=-p_1p_2}$, and we have to dismiss the former. Thus, let us assume that ${\Delta=-p_1^2p_2}$. Defining
\begin{equation}
\label{eamin}
a=\min \{p_1^2,p_2,(p_1^2+p_2)/4\},
\end{equation}
item~\ref{icoprime} of Proposition~\ref{psuit} implies that~$a$ is suitable for~$\Delta$.

Since ${-p_1^2p_2}$ is a discriminant,  ${-p_2}$ is a discriminant as well, and we have two possible cases.

\subsubsection*{Case~1: The discriminant $-p_2$ admits a suitable integer which is not a power of~$p_1$}

Item~\ref{idiv} of Proposition~\ref{psuit} implies that $-p_2$ admits a suitable prime ${q\ne p_1}$. Then ${(-p_2/q)=1}$ and ${p_2\ge 3q^2}$, which implies that  ${(\Delta/q)=1}$ and ${|\Delta|\ge 9p_2\ge 27q^2}$. Item~\ref{ikrone} of Proposition~\ref{psuit} implies that $q$ is suitable for~$\Delta$ as well.  

Now if ${p_2\ge 3p_1^4}$ then~$a$, defined in~\eqref{eamin}, satisfies 
$$
a=p_1^2\ne q, \qquad |\Delta|\ge 3a^3.
$$ 
Proposition~\ref{pflorian} implies that
${|\Delta/3|^{1/3} \ge a \ge 4|\Delta|^{1/2}/\log|\Delta|}$,
which is impossible when ${|\Delta|\ge 10^{11}}$. 

And if ${p_2\le 3p_1^4}$ then from ${p_2\ge 3q^2}$ we deduce 
${|\Delta|\ge 3q^3}$. Now
Corollary~\ref{ccoprime} implies that
${|\Delta/3|^{1/3} \ge q \ge 4|\Delta|^{1/2}/\log|\Delta|}$, which is again impossible. 

\subsubsection*{Case~2: Every integer suitable  for $-p_2$ is a power of~$p_1$} 
Item~\ref{idiv} of Proposition~\ref{psuit} implies that in this case the list of suitable integers for $-p_2$ consists of consecutive powers of~$p_1$:
$$
1,p_1, p_1^2, \ldots, p_1^\ell. 
$$
The suitable integer~$1$ occurs in only one triple in ${T_{-p_2}}$, and each of the suitable  integers  ${p_1, p_1^2,\ldots,p_1^\ell}$  occurs in exactly~$2$ triples. Hence ${h(-p_2)=2\ell+1}$.

On the other hand, from ${p_2^3\ge |\Delta|\ge 10^{11}}$ we deduce that ${p_2\ge 4000}$,  which implies that ${h(-p_2)>6}$ (the largest fundamental discriminant of class number not exceeding~$6$ is $-3763$, see~\cite[Table~4 on page~936]{Wa04}). Hence ${\ell\ge 3}$, or, equivalently,  ${p_1^3}$ must be suitable for~$-p_2$. This implies, in particular, that ${p_2\ge 3p_1^6}$. Hence~$a$,  defined in~\eqref{eamin}, is equal to~$p_1^2$. 
This shows that~$p_1^2$ is suitable for~$\Delta$.

We claim that ${p_1^3}$ is suitable for~$\Delta$ as well. Indeed, since~$p_1$ is suitable for~$-p_2$, there exist  ${b_1,c_1\in \Z}$ such that 
${b_1^2-4p_1c_1=-p_2}$ and ${0<b_1<p_1}$. Using ${p_2\ge 3p_1^6}$ we obtain 
$$
c_1=\frac{p_2+b_1^2}{4p_1}\ge \frac34 p_1^5. 
$$
Now a routine verification shows that
\begin{align*}
(p_1^3,b_1p_1,c_1)&\in T_\Delta &&\text{when ${p_1\nmid c_1}$},\\ 
(p_1^3,(2p_1-b_1)p_1,c_1-b_1+p_1)&\in T_\Delta 
&&\text{when ${p_1\mid c_1}$}.
\end{align*}
This proves that $p_1^3$ is suitable for~$\Delta$.

Thus, both $p_1^2$ and $p_1^3$ are suitable for~$\Delta$. Since ${p_2\ge 3p_1^6}$, we have ${|\Delta|\ge 3p_1^8}$. Hence ${|\Delta/3|^{1/4}\ge p_1^2}$. Proposition~\ref{pflorian} implies that 
$$
|\Delta/3|^{1/4}\ge p_1^2\ge 4|\Delta|^{1/2}/\log|\Delta|,
$$
which is impossible when ${|\Delta|\ge 10^{11}}$.

\bigskip

This completes the proof of Theorem~\ref{thfund}. \qed

\section{Primality of suitable integers}
\label{sprimal}

As before,~$\Delta$ denotes a trinomial discriminant unless the contrary is stated explicitly. In particular, ${|\Delta|>10^{11}}$ by Theorem~\ref{thsmall}, and, according to Theorem~\ref{thfund}, we have   ${\Delta=-p}$ or ${\Delta=-pq}$ where $p,q$ are distinct odd prime numbers.

As we have seen in  Corollary~\ref{ccoprime}, suitable integers for trinomial discriminants are prime numbers with some rare exceptions. It turns out that there are no exceptions at all. 

\begin{proposition}
\label{pprime}
Let~$\Delta$ be a trinomial discriminant and ${a>1}$ suitable for~$\Delta$. Then~$a$ is prime and satisfies ${a>4|\Delta|^{1/2}/\log|\Delta|}$. 
\end{proposition}

\begin{remark}
\label{rabig}
Since ${|\Delta|\ge 10^{11}}$, this implies, in particular, that ${a>10^4}$. 
\end{remark}



Before proving Proposition~\ref{pprime}, observe  that, in the case ${\Delta=-pq}$,  the primes $p,q$ are of the same order of magnitude up to a logarithmic factor.

\begin{proposition}
\label{ppqalmost=}
If  ${\Delta=-pq}$ then 
\begin{equation}
\label{epqsame}
\frac{4|\Delta|^{1/2}}{\log|\Delta|}< p,q < \frac{|\Delta|^{1/2}\log|\Delta|}{4}. 
\end{equation}
\end{proposition} 

\begin{proof}
It suffices to prove the lower estimate in~\eqref{epqsame}; the upper estimate will then follow automatically. Thus, let us assume that ${p<q}$ and prove that ${p>4|\Delta|^{1/2}/\log|\Delta|}$. 

Since ${|\Delta|>10^{11}}$ we have ${4|\Delta|^{1/2}/\log|\Delta|<|\Delta|^{1/2}/\sqrt3}$. Hence we may assume that ${p<q/3}$, in which case~$p$ is suitable for~$\Delta$ by  item~\ref{icoprime} of Proposition~\ref{psuit}. Proposition~\ref{psuitthree} implies that~$\Delta$ has a suitable integer other than~$1$ and~$p$, and Proposition~\ref{pflorian} implies that  ${p>4|\Delta|^{1/2}/\log|\Delta|}$. 
\end{proof}


\begin{proof}[Proof of Proposition~\ref{pprime}]
If ${\gcd(a,\Delta)=1}$ then Corollary~\ref{ccoprime} does the job. In particular, this completes the proof in the case ${\Delta=-p}$.  Now assume that ${\Delta=-pq}$ with ${p<q}$ and ${\gcd(a,\Delta)>1}$. Since ${a\le |\Delta/3|^{1/2}}$, the only possibility is  ${\gcd(a,\Delta)=p}$, and we claim that ${a=p}$.

Indeed, assume that ${a>p}$, and let~$\ell$ be a prime divisor of ${a/p}$. From ${a\le |\Delta/3|^{1/2}}$ and ${p>4|\Delta|^{1/2}/\log|\Delta|}$ we deduce 
$$
\ell<0.2\log|\Delta|  < 4|\Delta|^{1/2}/\log|\Delta| <p.
$$
Hence~$\ell$ is coprime with~$\Delta$, which implies that it  is suitable for~$\Delta$, see item~\ref{idiv} of Proposition~\ref{psuit}. Now Corollary~\ref{ccoprime} implies that ${\ell>4|\Delta|^{1/2}/\log|\Delta|}$, a	contradiction.   
\end{proof}

\section{A conditional result}
\label{sproofthgrh}

In this section we prove Theorem~\ref{thgrhintro}. Let us reproduce it here for convenience.

\begin{theorem}
\label{thgrh}
Assume GRH. Then a singular modulus of degree at least~$3$ cannot be a root of a trinomial with rational coefficients. In other words, GRH implies that trinomial discriminants do not exist. 
\end{theorem}

In this section, by the \textit{RH} we mean the \textit{Riemann Hypothesis} for the Riemann  $\zeta$-function, and by \textit{GRH}  the \textit{Generalized Riemann Hypothesis} for Dirichlet  $L$-functions.

Due to the results of the previous sections, Theorem~\ref{thgrh} is an easy consequence of  the following statement.

\begin{proposition}
\label{pourlam}
Assume GRH. 
Let ${m\ge 10^{10}}$ be an integer with  ${\omega(m)\le 2}$, and~$\chi$ a primitive odd real Dirichlet character modulo~$m$.  Then there exists a prime number~$p$ such that ${\chi(p)=1}$ and 
\begin{equation}
\label{eourlam}
p \le 4m^{1/2}/\log m. 
\end{equation}
\end{proposition}

Recall that a real character~$\chi$ is called \textit{odd} if ${\chi(-1)=-1}$. Restricting to odd characters is purely opportunistic here: the same argument, with very insignificant changes, applies to even real  primitive characters as well. But we apply estimate~\eqref{eourlam} only  to real odd characters, and making this assumption allows us to shorten the proof. The assumption ${\omega(m)\le 2}$ is of similar nature: it can be dropped, making the proof a bit more complicated, but this is unnecessary because we will apply~\eqref{eourlam} only to~$m$ with at most~$2$ prime divisors.

\begin{remark}
Unconditionally, the bound ${p\ll_\eps m^{1/4+\eps}}$ holds for every ${\eps>0}$; this is a classical result of Linnik and Vinogradov. Unfortunately, the implied constant in this estimate depends ineffectively on~$\eps$, because of ineffectiveness of Siegel's  bound for the exception real zero of $L(s,\chi)$. In Section~\ref{sallbutone} we imitate the Linnik-Vinogradov argument in the form given in~\cite{Po17}, but with Siegel's Theorem replaced by Tatuzawa's Theorem~\cite{Ta51}, obtaining this way an unconditional explicit upper bound for all but one trinomial discriminants. 
\end{remark}

\begin{proof}[Proof of Theorem~\ref{thgrh} (assuming Proposition~\ref{pourlam})] Let~$\Delta$ be a trinomial discriminant. 
Theorem~\ref{thsmall} implies that ${|\Delta|\ge 10^{10}}$. We apply Proposition~\ref{pourlam} with the character $(\Delta/\cdot)$, which is an odd real Dirichlet character $\mod|\Delta|$. Moreover, it is primitive because~$\Delta$ is fundamental, see Theorem~\ref{thfund}. Note also that ${\omega(|\Delta|) \le 2}$, again  by Theorem~\ref{thfund}.  We find (assuming GRH) a prime~$p$ satisfying ${(\Delta/p)=1}$ and ${p\le 4|\Delta|^{1/2}/\log|\Delta|}$, which contradicts Corollary~\ref{cdeltapone}. 
\end{proof}

The proof of Proposition~\ref{pourlam} is an adaptation  of the argument developed by Lamzouri et al. in~\cite{LLS15}. Their Theorem~1.4 implies, in our case, the estimate 
\begin{equation}
\label{elam}
p \le \max\left\{10^9, \left(\log m+\frac52(\log\log m)^2+6\right)^2\right\}. 
\end{equation}
Of course, it is asymptotically much sharper than~\eqref{eourlam}, but~\eqref{elam} is not suitable for our purposes because of the term $10^9$.

We prove Proposition~\ref{pourlam} in Subsection~\ref{ssproofpourlam}, after some preparatory work in Subsection~\ref{sslamzlems}.  

\subsection{Lemmas from~\cite{LLS15}}
\label{sslamzlems}
In this subsection we recall some technical lemmas  proved in~\cite{LLS15}, and give simplified versions of them. We use the notation of~\cite{LLS15} whenever possible; our only major deviation from the set-up of~\cite{LLS15} is that we denote the modulus by~$m$, while  it is usually denoted by~$q$ therein. 

For ${x>1}$ and a Dirichlet character~$\chi$ define
\begin{align*}
S(x)&= \sum_{1\le n\le x} \Lambda(n) \log \frac xn,&
S(x,\chi)&= \sum_{1\le n\le x} \chi(n)\Lambda(n)  \log \frac xn, \\
T(x)&=\sum_{1\le n\le x} \frac{\Lambda(n)}{n} \left( 1-\frac nx\right), & 
T(x,\chi)&=\sum_{1\le n\le x} \chi(n)\frac{\Lambda(n)}{n} \left( 1-\frac nx\right).
\end{align*}
Here $\Lambda(\cdot)$ is, of course, the von Mangoldt function.

Denote by~$\gamma$ the Euler–Mascheroni constant, and define 
\begin{equation}
\label{edefb}
B=  \frac 12\log (4\pi) - 1 -\frac{\gamma}{2}
= - 0.02309\ldots
\end{equation}
(see equation~(2.2) on \cite[page 2395]{LLS15}). 
The following is combination of Lemmas~2.1 and~2.4 from~\cite{LLS15}.

\begin{lemma}
\label{lzeta}
Assume RH.  Then for ${x>1}$ we have
\begin{align*}
S(x) &=x - (\log 2\pi) \log x  -1 + \sum_{k=1}^{\infty}\frac{1-x^{-2k}}{4k^2} +O_1\bigl(2 |B| (x^{1/2}+1)\bigr),\\
T(x) &=  \log x - (1+\gamma) + 
\frac{\log (2\pi)}{x} - \sum_{n=1}^{\infty} \frac{x^{-2n-1}}{2n(2n+1)} 
+ O_1\left( \frac{2|B|}{x^{1/2}}\right),
\end{align*}
where~$B$ is defined in~\eqref{edefb}. 
\end{lemma}

Recall (see Subsection~\ref{ssconv}) that ${X=O_1(Y)}$ means that ${|X|\le Y}$.

We will use the following simplified version of this lemma for large~$x$. 

\begin{lemma}
\label{lzetasimple}
In the set-up of Lemma~\ref{lzeta}, for ${x\ge 100}$ we have
\begin{equation}
\label{eoursxsmall}
S(x) =x - (\log 2\pi) \log x  +O_1\bigl(0.16x^{1/2}\bigr)\\ 
\le 1.02x.
\end{equation}
For ${x\ge 10^{4}}$ we have 
\begin{align}
\label{eours}
S(x) &=x - (\log 2\pi) \log x  +O_1\bigl(0.06x^{1/2}\bigr),\\
\label{eourt}
T(x) &=  \log x - (1+\gamma) 
+ O_1\left( \frac{0.07}{x^{1/2}}\right)\le \log x-1.576.
\end{align}
\end{lemma}
The proof of this lemma is left out, being an easy calculation.

\bigskip

For ${x>1}$ define 
\begin{align*}
\tilE_1(x) &=\frac{\pi^2}{8}- \left(\log 2+\frac\gamma2\right) \log x -\sum_{k=0}^{\infty} \frac{x^{-2k-1}}{(2k+1)^2} ,\\
E_1(x) &= -\sum_{k=0}^{\infty} \frac{x^{-2k-2}}{(2k+1)(2k+2)} -\frac{\gamma}{2}\Big(1-\frac 1x\Big) + \frac{\log 2}{x}.
\end{align*}
The next lemma combines Lemmas~2.2 and~2.3 from~\cite{LLS15}, in the special case of odd real characters.

\begin{lemma}
\label{ll}
Let ${m\ge 3}$ be an integer and~$\chi$ be a primitive real odd Dirichlet character modulo~$m$. 
Assume GRH.    
Then for ${x>1}$ we have 
\begin{equation}
\label{eschilamz}
S(x,\chi) = R(\chi)\bigl(\log x+ O_1(2x^{1/2}
+2)\bigr)
+ \frac 12 \Big(\log \frac m{\pi}\Big) \log x + \tilE_1(x),
\end{equation}
where
$$
R(\chi)=
  \left(1+\frac{1}{x}+O_1\left(\frac{2}{x^{1/2}}\right)\right)^{-1}
\left(\frac 12\left(1-\frac 1x\right)\log \frac m{\pi} - T(x,\chi) + E_1(x)\right).
$$
\end{lemma}

Note that we denote by $R(\chi)$ the quantity ${|\Re (B(\chi))|}$ from~\cite{LLS15}. 

We again give a simplified version (of the lower bound only, we do not need the upper bound). 

\begin{lemma}
In the set-up of Lemma~\ref{ll} assume that 
$$
m\ge 10^{10}, \qquad 10^4\le x\le 0.2m^{1/2}.
$$
Then 
\begin{equation}
\label{eschiours}
S(x,\chi)\ge -2.1x^{1/2}(\log m-4) - (0.53\log m-1.9)\log x. 
\end{equation}
\end{lemma}

\begin{proof}
We have ${E_1(x) \le -0.288}$ and 
\begin{equation*}
|T(x,\chi)|\le T(x) \le  \log x-1.576 \le \frac12\log m -3.185,
\end{equation*}
see~\eqref{eourt}. Hence 
\begin{align}
\frac 12\left(1-\frac 1x\right)\log \frac m{\pi} - T(x,\chi) + E_1(x) &\le \log m -\frac12\log\pi -3.185-0.288 \nonumber\\
&\le \log m-4,\nonumber\\
\label{erchiest}
R(\chi)&\le 1.021(\log m-4). 
\end{align}
Furthermore, we have ${\tilE_1(x)\ge -(\gamma/2+\log 2)\log x}$. Hence
\begin{equation}
\label{esomeest}
\frac 12 \Big(\log \frac m{\pi}\Big) \log x + \tilE_1(x) \ge \left(\frac12\log m-\log(2\pi)-\frac\gamma2\right)\log x. 
\end{equation}
Substituting~\eqref{erchiest} and~\eqref{esomeest} into~\eqref{eschilamz}, we obtain~\eqref{eschiours}. 
\end{proof}

Finally,  the following  is (a consequence of) Lemma~3.1 from~\cite{LLS15} (which is unconditional, unlike the previous lemmas).

\begin{lemma}
\label{lnoncoprime}
Let ${m\ge 3}$ be an integer and ${x\ge 2}$ be a real number.  Then
$$
\sum_{\substack{1\le n\leq x\\(n,m)>1}}\Lambda(n)\log(x/n)
\le \frac{1}{2}\omega(m)(\log x)^2.
$$
\end{lemma}

\subsection{Proof of Proposition~\ref{pourlam}}
\label{ssproofpourlam}

Assume the contrary: ${\chi(p) \ne 1}$ for every ${p\le 4m^{1/2}/\log m}$. Since~$\chi$ is a real character,  this implies that 
\begin{equation*}
-S(x,\chi) \ge S(x) -\sum_{p\le x^{1/2}}\log p\log (x/p^2)- \sum_{\substack{1\le n\leq x\\(n,m)>1}}\Lambda(n)\log(x/n). 
\end{equation*}
where we set ${x=4m^{1/2}/\log m}$. Since ${m\ge 10^{10}}$, we have ${10^4\le x\le 0.2m^{1/2}}$, which means that we may use estimates~\eqref{eours} and~\eqref{eschiours}. We may also use~\eqref{eoursxsmall} with~$x$ replaced by~$x^{1/2}$, which gives
$$
\sum_{p\le x^{1/2}}\log p\log (x/p^2)\le 2S(x^{1/2}) \le 2.04x^{1/2}.
$$
Finally, Lemma~\ref{lnoncoprime} and the assumption ${\omega(m)\le 2}$ imply that 
$$
\sum_{\substack{1\le n\leq x\\(n,m)>1}}\Lambda(n)\log(x/n)\le (\log x)^2 \le \log(0.2m^{1/2})\log x\le (0.5\log m-1.6)\log x. 
$$
Combining all these estimates, we obtain
{\small
\begin{align*}
2.1x^{1/2}(\log m-4) + (0.53\log m-1.9)\log x &\ge x - (\log 2\pi) \log x-0.06x^{1/2} \\
&\hphantom{\ge}- 2.04x^{1/2}
- (0.5\log m-1.6)\log x,
\end{align*}
}%
which can be re-written as 
\begin{align*}
2.1x^{1/2}(\log m-3) + (1.03\log m-1.6)\log x &\ge x.  
\end{align*}
When ${x\ge 10^4}$ the left-hand side does not exceed ${(2.2\log m-6)x^{1/2}}$, which implies the inequality ${x^{1/2}\le 2.2\log m-6}$. Substituting ${x=4m^{1/2}/\log m}$, we obtain 
$$
m^{1/4}\le (1.1\log m-3)(\log m)^{1/2}. 
$$
This inequality is impossible when ${m\ge 10^{10}}$. 
\qed

\section{Bounding all but one trinomial discriminants}
\label{sallbutone}
In this section we prove the following theorem. 

\begin{theorem}
\label{thoneexception}
There exists at most one trinomial discriminant~$\Delta$ satisfying ${|\Delta|\ge 10^{160}}$. 
\end{theorem}

Call a positive integer~$m$ \textit{coarse}\footnote{as opposed to \textit{smooth}} if~$m$ is either prime or a product of two distinct primes each exceeding ${m^{3/8}\log m}$. We deduce Theorem~\ref{thoneexception} from the following statement. 

\begin{theorem}
\label{thtatubur}
With at most one exception, every coarse integer  ${m>10^{160}}$ has the following property. Let~$\chi$ be  a primitive real Dirichlet 
character $\mod m$. Then there exists a prime ${p\le 4m^{1/2}/\log m}$ such that ${\chi(p)=1}$. 
\end{theorem}

\begin{proof}[Proof of Theorem~\ref{thoneexception} (assuming Theorem~\ref{thtatubur})]
Let~$\Delta$ be a trinomial discriminant satisfying ${|\Delta|\ge 10^{160}}$.  Then ${m=|\Delta|}$ is either prime or product of two distinct primes both exceeding ${4m^{1/2}/\log m}$, see Proposition~\ref{ppqalmost=}. Since ${m\ge 10^{160}}$, it must be coarse. Let $\chi(\cdot)$ be the Kronecker $(\Delta/\cdot)$; then, unless~$m$ is the exceptional one from Theorem~\ref{thtatubur}, there exists ${p\le 4m^{1/2}/\log m}$ such that ${\chi(p)=1}$, contradicting Corollary~\ref{cdeltapone}. 
\end{proof}

Theorem~\ref{thtatubur} will be proved in Subsection~\ref{sspol}, after we establish, in Subsection~\ref{ssburg}, an explicit version of the Burgess estimate for coarse moduli.

\subsection{Explicit Burgess for coarse moduli}
\label{ssburg}
Everywhere in this subsection~$m$ is a positive integer and~$\chi$ a primitive Dirichlet character $\mod m$.
For ${M,N\in \Z}$ with ${N>0}$  denote  
$$
S(N,M)=\sum_{M<n\le M+N}\chi(n). 
$$
A classical result of Burgess~\cite{Bu62,Bu63} implies that 
${|S(M,N)|\ll_\eps N^{1/2} m^{3/16+\eps}}$. We need a version of this inequality explicit in all parameters. Such a version is available in the case of prime modulus \cite{Bo06,IK04,Tr15}, but we need a slightly more general version of it, for coarse moduli, as defined in the beginning of Section~\ref{sallbutone}. 

\begin{theorem}
\label{thburg}
Let ${m>10^{11}}$ be a coarse integer.  
Then for every $M,N$ as above 
we have 
\begin{equation*}
|S(N,M)|<  10N^{1/2}m^{3/16}(\log m)^{1/2}.
\end{equation*}
\end{theorem}

Note that we did not try to optimize the numerical constant~$10$. Probably, sharper constants are possible, as the work of Booker~\cite{Bo06} and Treviño~\cite{Tr15} suggests. 

\begin{remark}
After this article was submitted, fully explicit versions of the Burgess inequality with arbitrary composite moduli emerged, see \cite[Theorem~1.1]{Bo20} or \cite[Theorem~1.1 and Corollary~1.2]{JKL21}. However, for the sake of consistency, we prefer to use our Theorem~\ref{thburg}. 
\end{remark}

The following lemma is quite standard, but we did not find a suitable reference.

\begin{lemma}
\label{lweil}
Assume that~$m$ is square-free, and let ${f(x)\in \Z[x]}$ be a polynomial with the following property: there exists ${b\in \Z}$ such that 
$$
m\mid f(b), \qquad \gcd(f'(b),m)=1.
$$
Then the sum 
$$
S_\chi(f)=\sum_{x\bmod m}\chi(f(x))
$$
satisfies ${|S_\chi(f)|\le (\mu-1)^{\omega(m)}m^{1/2}}$, where~$\mu$ is the number of distinct roots of~$f$ modulo~$m$. 
\end{lemma}

\begin{proof}
Let ${m=p_1\cdots p_k}$ be the prime factorization of~$m$ (recall that~$m$ is square-free), and set ${m_i=m/p_i}$. Our character~$\chi$ has a unique presentation as ${\chi_1\cdots \chi_k}$, where each~$\chi_i$ is a primitive character $\mod p_i$. Then 
$$
S_\chi(f)=S_{\chi_1}(f_1)\cdots S_{\chi_k}(f_k),
$$
where ${f_i(x)=f(m_i x)}$; see, for instance, equation~(12.21) in~\cite{IK04}. Since each~$f_i$ has a simple root modulo~$p_i$, the Hasse-Weil bound 
${|S_{\chi_i}(f_i)|\le (\mu-1)p_i^{1/2}}$
applies (see, for instance, \cite{Sc76}, Theorem~2C${}'$ on page~43). The result follows. 
\end{proof}

\begin{proof}[Proof of Theorem~\ref{thburg}]
Denote 
${E(N) =\max\{|S(N,M)|: M\in \Z\}}$. 
We want to prove that 
\begin{equation}
\label{eburg}
E(N)< 10N^{1/2}m^{3/16}(\log m)^{1/2}.
\end{equation}
We follow rather closely the argument from the book of Iwaniec and Kowalski \cite[pages 327--329]{IK04}, where we set ${r=2}$.  In particular, we will use induction in~$N$.

If ${N< 100m^{3/8}\log m}$ 
then~\eqref{eburg} follows from  the trivial estimate ${E(N)\le N}$, 
and if ${N\ge m^{5/8}\log m}$ then~\eqref{eburg} follows from the Pólya-Vinogradov inequality ${E(N)\le 6m^{1/2}\log m}$, see \cite[Theorem~12.5]{IK04}. 
Hence we may assume in the sequel that 
\begin{equation}
\label{eframen}
100m^{3/8}\log m\le N<  m^{5/8}\log m.
\end{equation}
We fix positive integers $A,B$, to be specified later, such that 
\begin{equation}
\label{eabh}
3\le A,B<m^{3/8}\log m, \qquad AB<N. 
\end{equation}
Since~$m$ is coarse, 
\begin{equation}
\label{ecoprime}
\gcd(a,m)=1 \qquad (1\le a\le \max\{A,B\}). 
\end{equation}
For a residue class $x\bmod m$ denote by $\nu(x)$ the number of presentation of~$x$ as ${\bar a n}$, where 
$$
1\le a\le A, \qquad M<n\le M+N, 
$$
and~$\bar a$ denote the inverse of~$a$ modulo~$m$.

Arguing as in \cite[page 327]{IK04}, we find 
\begin{equation}
\label{eik}
|S(M,N)| \le V/H+2E(H),
\end{equation}
where
$$
H=AB, \qquad V= \sum_{x\bmod m} \nu(x)\left|\sum_{1\le b\le B}\chi(x+b)\right|. 
$$
Using Hölder's inequality, we estimate 
$$
V\le V_1^{1/2}V_2^{1/4}W^{1/4}, 
$$
where 
$$
V_1=\sum_{x\bmod m}\nu(x), \quad V_2=\sum_{x\bmod m}\nu(x)^2, \quad 
W=\sum_{x\bmod m}\left|\sum_{1\le b\le B}\chi(x+b)\right|^4. 
$$
We have the following estimates:
\begin{equation}
\label{eests}
V_1\le NA,\qquad
V_2\le 8NA\left(\frac{NA}m+\log (3A)\right),\qquad
W\le 9B^4m^{1/2}+3B^2m.  
\end{equation}
Now we are going to complete the proof, assuming them.  Estimates~\eqref{eests} themselves will be proved afterwards. 

Set 
$$
A=\lfloor 0.04Nm^{-1/4}\rfloor, \qquad B=\lfloor m^{1/4}\rfloor. 
$$
From the hypotheses ${m>10^{11}}$ and the inequality~\eqref{eframen} we deduce that~\eqref{eabh} indeed holds. Moreover,  ${H=AB}$ satisfies 
${0.03N< H\le 0.04N}$.

Since ${H<N}$, estimate~\eqref{eburg} holds true, by induction, with~$N$ replaced by~$H$:
$$
E(H)< 10H^{1/2}m^{3/16}(\log m)^{1/2} \le 2N^{1/2}m^{3/16}(\log m)^{1/2}.
$$
Now let us estimate~$V$. 
We have clearly 
$$
V_1\le NA\le 0.04N^2m^{-1/4}, \qquad W\le 12m^{3/2}. 
$$
Furthermore, using~\eqref{eframen} and ${m>10^{11}}$, we obtain 
$$
\log (3A)\le 0.5\log m, \qquad NA/m\le 0.04(\log m)^2. 
$$
It follows that ${V_2\le 0.5NA(\log m)^2}$. 
We obtain 
\begin{align*}
V&\le (NA)^{1/2}(0.5NA(\log m)^2)^{1/4}(12m^{3/2})^{1/4}\\
&< 0.14 N^{3/2}m^{3/16}(\log m)^{1/4}.
\end{align*}
Substituting all this in~\eqref{eik}, we obtain 
\begin{align*}
|S(M,N)| &< \frac{0.14 N^{3/2}m^{3/16}(\log m)^{1/4}}{0.03N} +2\cdot 2 N^{1/2}m^{3/16}(\log m)^{1/2} \\
&<9 N^{1/2}m^{3/16}(\log m)^{1/2},
\end{align*}
as wanted. 

We are left with the estimates from~\eqref{eests}. The estimate ${V_1\le NA}$ is obvious. The estimate for~$V_2$ is Lemma~12.7 from \cite[page~328]{IK04}. The only difference is that in~\cite{IK04}~$m$ is a prime number (and denoted~$p$). However, it is only needed therein  that every integer between~$1$ and~$A$ is co-prime with~$m$, which is true in our case, see~\eqref{ecoprime}.

Finally, let us prove the estimate for~$W$. The proof is very similar to that of \cite[Lemma~12.8]{IK04}. We write 
$$
W=\sum_{1\le b_1\ldots b_4\le B}\sum_{x\bmod m}\chi\bigl((x+b_1)(x+b_2)\bigr)\bar\chi\bigl((x+b_3)(x+b_4)\bigr). 
$$ 
Note that ${b_i\ne b_j}$ implies that ${\gcd(b_i-b_j,m)=1}$, 
see~\eqref{ecoprime}. Now if a quadruple $(b_1, \ldots,b_4)$ has the property that
\begin{equation}
\label{eprop}
\text{some~$b_i$ is distinct from all other~$b_j$},
\end{equation}
then 
$$
\left|
\sum_{x\bmod m}\chi\bigl((x+b_1)(x+b_2)\bigr)\bar\chi\bigl((x+b_3)(x+b_4)\bigr)\right|\le 9m^{1/2}
$$
by Lemma~\ref{lweil} applied to the polynomial 
$$
f(x)=(x+b_1)(x+b_2)\bigl((x+b_3)(x+b_4)\bigr)^{\ph(m)-1}.
$$
(Here, the~$b$ from Lemma~\ref{lweil} is  the~$b_i$
that is distinct from all other~$b_j$.)
And a simple combinatorial argument shows that  exactly ${3B^2-2B}$ quadruples do not satisfy~\eqref{eprop}. Hence 
$$
W\le 9m^{1/2}(B^4 - 3B^2+2B)+m(3B^2-2B),
$$ 
which is slightly sharper than wanted. The theorem is proved. 
\end{proof}

\subsection{Proof of Theorem~\ref{thtatubur}}
\label{sspol}


Set 
\begin{equation}
\label{edefrho}
\rho(n)=\sum_{d\mid n}\chi(d). 
\end{equation}
Note that $\rho(\cdot)$ is a multiplicative function.

The following statement is a version of Proposition~3.1 from Pollack~\cite{Po17}.

\begin{proposition}
\label{psum}
Let~$m$   and~$\chi$ be as in Theorem~\ref{thburg} (in particular,~$m$ is coarse). 
Then for ${x \ge 10^4m^{3/8}\log m}$ we have 
\begin{equation}
\label{esum}
\sum_{1\le n\le x} \rho(n)=xL(1,\chi) +O_1\bigl(50m^{1/8}(\log m)^{1/3} x^{2/3}\bigr). 
\end{equation}
\end{proposition}
(Recall that ${A=O_1(B)}$ means ${|A|\le B}$.)
\begin{proof}
Set
${\Theta=10m^{3/16}(\log m)^{1/2}}$,
so that~\eqref{eburg} can be written as  
\begin{equation}
\label{eburgnew}
\left|\sum_{M<n\le M+N}\chi(n)\right|\le \Theta N^{1/2}. 
\end{equation} 
Let~$y$ be a real number satisfying ${1\le y<x}$, to be specified later, and set ${z=x/y}$. 
Intuitively, one should think of~$y$ as ``large'' (not much smaller than~$x$) and~$z$ ``small''. As in~\cite{Po17}, we use the ``Dirichlet hyperbola formula''
\begin{equation*}
\sum_{1\le n\le x} \rho(n)= \sum_{1\le d \le y}\chi(d)\sum_{1\le e\le x/d}1+ \sum_{1\le e \le z}\sum_{1\le d\le x/e}\chi(d)-\sum_{1\le d \le y}\chi(d)\sum_{1\le e\le z}1. 
\end{equation*} 
Here the first double sum will give the main contribution, while the second and the third double sums will be absorbed in the error term.

Using~\eqref{eburgnew}, we estimate the last two double sums:
\begin{align*}
\Bigl|\sum_{1\le e \le z}\sum_{1\le d\le x/e}\chi(d)\Bigr|&\le \Theta x^{1/2}\sum_{1\le e \le z}\frac{1}{e^{1/2}}\le 2\Theta x^{1/2}(1+z^{1/2})\le 4\Theta x^{1/2}z^{1/2},\\ 
\Bigl|\sum_{1\le d \le y}\chi(d)\sum_{1\le e\le z}1\Bigr|&\le \Theta y^{1/2}z = \Theta x^{1/2}z^{1/2}. 
\end{align*}
For the first double sum we have the expression 
\begin{equation}
\sum_{1\le d \le y}\chi(d)\sum_{1\le e\le x/d}1= \sum_{1\le d \le y}\chi(d)\left\lfloor\frac{x}{d}\right\rfloor=xL(1,\chi)+R_1+R_2,
\end{equation}
where 
\begin{equation*}
R_1 =\sum_{1\le d \le y}\chi(d)\left(\left\lfloor\frac{x}{d}\right\rfloor-\frac{x}{d}\right),\qquad
R_2=-x\sum_{d>y}\frac{\chi(d)}{d}. 
\end{equation*}
We have clearly ${|R_1|\le y}$. To estimate~$R_2$ we use partial summation. For an integer ${k> y}$  set
$
{S_k =\sum_{y<n\le k}\chi(n)} 
$.
Using~\eqref{eburgnew} we estimate
$$
|S_k|\le \Theta(k-y)^{1/2}\le \Theta k^{1/2}.
$$ 
Hence 
$$
|R_2|= x\Bigl|\sum_{k>y}S_k\Bigl(\frac1k-\frac1{k+1}\Bigr)\Bigr|\le \Theta x \sum_{k>y}\frac{1}{k^{1/2}(k+1)}\le 3\Theta xy^{-1/2}. 
$$
Combining all these estimates, we obtain 
\begin{equation*}
\sum_{1\le n\le x} \rho(n)=xL(1,\chi) +R
\end{equation*}
with 
\begin{equation*}
 |R|\le 4\Theta x^{1/2}z^{1/2}+\Theta x^{1/2}z^{1/2}+y+3\Theta xy^{-1/2}= 8\Theta xy^{-1/2}+y.
\end{equation*}
We set the ``optimal'' ${y=4\Theta^{2/3}x^{2/3}}$. Our assumption ${x\ge 10^4m^{3/8}\log m}$ implies that indeed ${1\le y <x}$. We obtain 
$$
|R|\le 8\Theta^{2/3}x^{2/3} < 50 m^{1/8}(\log m)^{1/3} x^{2/3},
$$
as wanted. 
\end{proof}

The following lemma is  the classical theorem of Tatuzawa~\cite{Ta51}. 

\begin{lemma}
\label{ltatu}
Let ${0<\eps<1/11.2}$. Then, with at most one exception,   for every positive integer ${m>e^{1/\eps}}$ the following holds. Let~$\chi$ be a primitive real character $\mod m$. Then 
${L(1,\chi)>0.655\eps m^{-\eps}}$.
\end{lemma}

\begin{proof}[Proof of Theorem~\ref{thtatubur}]
We assume that ${m\ge 10^{160}}$ is coarse and not the exceptional one from Lemma~\ref{ltatu}, where we set ${\eps=1/360}$. 
Set ${x=4m^{1/2}/\log m}$. 
We have ${m>e^{360}}$ and    ${x \ge 10^4m^{3/8}\log m}$. Hence, combining Proposition~\ref{psum} and Lemma~\ref{ltatu} with ${\eps=1/360}$, we obtain
\begin{equation}
\label{esumlower}
\sum_{1\le n\le x} \rho(n)\ge \frac{0.655}{360}xm^{-1/360} - 50 m^{1/8}(\log m)^{1/3} x^{2/3}. 
\end{equation}

On the other hand, if~$p$ is such that ${\chi(p)\ne1}$ then  
\begin{equation}
\label{echipk}
\rho(p^k)= \frac{1-\chi(p)^{k+1}}{1-\chi(p)}=
\begin{cases}
1,\chi(p)=0,\\
1,\chi(p)=-1, \ 2\mid k,\\
0, \chi(p)=-1, \ 2\nmid k.
\end{cases}
\end{equation}
Now assume that   ${\chi(p)\ne 1}$ for all primes ${p\le x}$. We have two cases.

If~$m$ is prime, then, by~\eqref{echipk} and the multiplicativity of $\rho(\cdot)$, 
for ${1\le n <x}$ we have ${\rho(n)=1}$ when~$n$ is a full square, and ${\rho(n)=0}$ otherwise.

Now assume that~$m$ is a product of two primes,~$\ell$ being the smallest of them; in particular, ${\rho(\ell)=1}$. Since  ${x<m^{1/2}}$, we  have ${\gcd(n,m)\in \{1,\ell\}}$  for any~$n$ in the range
 ${1\le n <x}$. It follows that ${\rho(n)=1}$ when~$n$ is a full square or~$\ell$ times a full square, and ${\rho(n)=0}$ otherwise. 

Thus, in any case
$$
\sum_{1\le n\le x} \rho(n)\le 2x^{1/2}. 
$$
Together with~\eqref{esumlower} this implies
 ${m^{7/180}\le 30000(\log m)^{2/3}}$. This inequality is contradictory for ${m\ge 10^{160}}$. 
\end{proof}

\begin{remark}
At present, numerical refinements of Tatuzawa's theorem are available, see~\cite{Ch07} and the references therein. 
In particular, using the main result of~\cite{Ch07}, one can reduce the numerical bound in Theorem~\ref{thtatubur} to $10^{140}$.  
\end{remark}

\section{The quantities $h(\Delta)$, $\rho(\Delta)$ and  $N(\Delta)$}
\label{shrhon}

This section is preparatory for the ``signature theorem'', proved in Section~\ref{ssign}. 
As before,~$\Delta$ denotes a trinomial discriminant unless the contrary is stated explicitly. 
The following three quantities will play a crucial role in the sequel:

\begin{itemize}
\item
$h(\Delta)$, the class number;

\item
$\rho(\Delta)$, the largest absolute value of a non-dominant singular modulus of discriminant~$\Delta$;

\item 
$N(\Delta)$, the absolute norm ${|\norm_{\Q(x)/\Q}(x)|}$, where~$x$ is a singular modulus of discriminant~$\Delta$; it clearly depends only on~$\Delta$ and not on the particular choice of~$x$. 
\end{itemize}

For an arbitrary (not necessarily trinomial) discriminant we have upper estimates 
$$
h(\Delta) \le \pi^{-1}|\Delta|^{1/2}(2+\log|\Delta|),  \qquad \rho(\Delta) \le e^{\pi|\Delta|^{1/2}/2}+2079
$$
(for the first one see, for instance, Theorems~10.1 and~14.3 in \cite[Chapter~12]{Hu82}). It turns out that  for trinomial discriminants one can do much better.

\begin{proposition}
\label{pcn&rho}
Let~$\Delta$ be a trinomial discriminant. Then 
\begin{equation}
\label{ecn&rho}
101\le h(\Delta)< 3|\Delta|^{1/2}/\log|\Delta|, \qquad 700|\Delta|^{-3}\le \rho(\Delta)< |\Delta|^{0.8}.
\end{equation}
\end{proposition}

An immediate application is the following much sharper form of  Corollary~\ref{cprinceq} (the ``principal inequality''). 

\begin{corollary}[refined ``principal inequality'']
\label{cprinceqrefined}
Let~$\Delta$ be a trinomial discriminant of signature $(m,n)$  and~$x_1,x_2$  non-dominant singular moduli of discriminant~$\Delta$. 
Then 
\begin{align*}
|1-(x_2/x_1)^n|&\le e^{(m-n)(-\pi|\Delta|^{1/2}+\log|\Delta|)},\\
\bigl||x_1|-|x_2|\bigr|&\le \rho(\Delta)e^{(m-n)(-\pi|\Delta|^{1/2}+\log|\Delta|)},\\
\bigl||x_1|-|x_2|\bigr|&\le \rho(\Delta)e^{-\pi|\Delta|^{1/2}+\log|\Delta|},\\ 
\bigl||x_1|-|x_2|\bigr|&\le e^{-\pi|\Delta|^{1/2}+2\log|\Delta|}. 
\end{align*}
\end{corollary}

For subsequent applications we need to estimate the product ${h(\Delta)\log\rho(\Delta)}$. 
Proposition~\ref{pcn&rho} implies the estimate 
$$
-9|\Delta|^{1/2}<h(\Delta)\log\rho(\Delta)\le 2.4|\Delta|^{1/2}. 
$$
However, this is insufficient for us: we need an estimate of the shape $o(|\Delta|^{1/2})$ on both sides. 

\begin{proposition}
\label{phlogrho}
Let~$\Delta$ be a trinomial discriminant. Then 
$$
-31|\Delta|^{1/2}/\log|\Delta| <h(\Delta)\log\rho(\Delta)< 120|\Delta|^{1/2}/\log|\Delta|. 
$$
\end{proposition}

Finally, we need an estimate for $N(\Delta)$. It is known that ${N(\Delta)>1}$ for any discriminant~$\Delta$, see~\cite{BHK20,Li21}. For trinomial discriminants one can say much more.

\begin{proposition}
\label{pn}
Let~$\Delta$ be a trinomial discriminant.  Then
\begin{equation}
\label{enasymp}
\log N(\Delta)=\pi|\Delta|^{1/2}+(h(\Delta)-1)\log\rho(\Delta)+O_1(e^{-\pi|\Delta|^{1/2}+2\log|\Delta|}).
\end{equation}
In particular, 
\begin{equation}
\label{enlower}
\log N(\Delta)>\pi|\Delta|^{1/2}-32|\Delta|^{1/2}/\log|\Delta|. 
\end{equation}
\end{proposition}

Propositions~\ref{pcn&rho},~\ref{phlogrho} and~\ref{pn} are proved in the subsequent subsections. Since~$\Delta$ is fixed, we may omit it in the sequel and write simply~$h$,~$\rho$ and~$N$. 

\subsection{Proof of Proposition~\ref{pcn&rho}}




According to Proposition~\ref{pprime},  all suitable integers except~$1$ are prime numbers not exceeding ${|\Delta/3|^{1/2}}$. Each of them occurs in most~$2$ triples ${(a,b,c)\in T_\Delta}$ (see the proof of Proposition~\ref{psuitthree}). Hence ${h  \le 1+2\pi(|\Delta/3|^{1/2})}$. 
Theorem~2 in \cite[page~69]{RS62} states that 
\begin{equation}
\label{epiappr}
\frac X{\log X-1/2}\le\pi(X) \le \frac X{\log X-3/2} \qquad (X\ge 67).
\end{equation}
Since ${|\Delta|\ge 10^{11}}$, we can apply this with ${X=|\Delta/3|^{1/2}}$.  
We obtain
$$
h \le 1+ 4\frac{|\Delta/3|^{1/2}}{\log|\Delta/3|-3}<3\frac{|\Delta|^{1/2}}{\log|\Delta|},
$$
which proves the upper estimate for $h $.

The lower bound for~$h$ follows from the work of Watkins~\cite{Wa04}, which implies that a fundamental discriminant~$\Delta$ with   ${h(\Delta)\le 100}$ satisfies ${|\Delta|\le 2383747}$: see Table~4 on page~936 of~\cite{Wa04}. But ${|\Delta|\ge 10^{11}}$, a contradiction. 

Now let~$x$ be a non-dominant singular modulus of discriminant~$\Delta$ and~$a$ the corresponding suitable integer. Since~$x$ is not dominant, we have ${a>1}$, which implies that ${a>4|\Delta|^{1/2}/\log|\Delta|}$. Hence 
$$
|x|\le e^{\pi|\Delta|^{1/2}/a}+2079 < e^{(\pi/4)\log|\Delta|}+2079<|\Delta|^{0.8},
$$
which proves the upper estimate for $\rho $.

As for the lower  estimate ${\rho \ge 700|\Delta|^{-3}}$, it holds true for any discriminant ${\Delta\ne -3}$, not only trinomial discriminants, due to the following lemma.

\begin{lemma}
\label{lbfz}
Let~$x$ be a singular modulus of discriminant ${\Delta\ne -3}$ (not necessarily trinomial). Then ${|x|\ge 700|\Delta|^{-3}}$. 
\end{lemma}

For the proof see \cite[Corollary~5.3]{BFZ20}. 

\subsection{Proof of Proposition~\ref{phlogrho}: the upper estimate}

Let ${1<a_1<\ldots<a_k}$ be all suitable integers for~$\Delta$. Since~$1$ occurs in one triple from~$T_\Delta$ and each~$a_i$ in at most two triples, we have 
$$
h \le 2k+1 \le 2\pi(a_k)+1 
$$
Now note that ${a_k\le 5a_1}$ by Proposition~\ref{ptight}, and ${a_1\ge 10^4}$, see Remark~\ref{rabig}. 
Hence we may use~\eqref{epiappr} with ${X=5a_1}$. We obtain
$$
h \le 2\frac{5a_1}{\log(5a_1)-3/2}+1< 10\frac{a_1}{\log a_1}.
$$
Now, using that ${4|\Delta|^{1/2}/\log|\Delta|<a_1<|\Delta/3|^{1/2}}$, we obtain
$$
h \log\rho <10\frac{a_1}{\log a_1} \log\bigl(e^{\pi|\Delta|^{1/2}/a_1}+2079\bigr)
< 50 \frac{|\Delta|^{1/2}}{\log a_1}<120\frac{|\Delta|^{1/2}}{\log|\Delta|}, 
$$
as wanted. 

(Note that our argument is quite loose:  the numerical constant~$120$ can be easily improved.)

\subsection{Proof of Proposition~\ref{phlogrho}: the lower estimate}

We need a lemma. Recall (see Section~\ref{sdom}) that~$\calF$ denotes the standard fundamental domain, and that 
$$
\zeta_3=\frac{-1+\sqrt{-3}}{2}, \qquad \zeta_6=\frac{1+\sqrt{-3}}{2}.
$$

\begin{lemma}
\label{lcat}
For  ${z\in \calF}$ the following is true. 
\begin{enumerate}
\item
If ${\min\{|z-\zeta_3|,|z-\zeta_6|\}\ge 10^{-3}}$  then ${|j(z)|\ge 4.4\cdot10^{-5}}$.

\item
If ${\min\{|z-\zeta_3|,|z-\zeta_6|\}\le 10^{-3}}$ then 
$$
|j(z)|\ge 44000\min\{|z-\zeta_3|,|z-\zeta_6|\}^3.
$$
\end{enumerate}
\end{lemma}

The proof can be found in~\cite[Proposition~2.2]{BLP16}. 

Now we are ready to prove the lower estimate from Proposition~\ref{phlogrho}. 
We consider two cases.
Assume first that there exists ${(a,b,c)\in T_\Delta}$ with ${a>1}$ such that ${\tau=(b+\sqrt\Delta)/2a}$ satisfies 
$$
\min\{|\tau-\zeta_3|,|\tau-\zeta_6|\}\ge 10^{-3}. 
$$
Lemma~\ref{lcat} implies that  the non-dominant singular modulus ${x=j(\tau)}$ satisfies ${|x|\ge 4.4\cdot10^{-5}}$. Hence ${\rho \ge 4.4\cdot10^{-5}}$. Using Proposition~\ref{pcn&rho}, we obtain
$$
h \log\rho  \ge 3\frac{|\Delta|^{1/2}}{\log|\Delta|}\log(4.4\cdot10^{-5}) > -31 \frac{|\Delta|^{1/2}}{\log|\Delta|},
$$
as wanted.

Now assume that 
for every ${(a,b,c)\in T_\Delta}$ with ${a>1}$ the number~$\tau$ as above satisfies 
${\min\{|\tau-\zeta_3|,|\tau-\zeta_6|\}<10^{-3}}$. 
Let~$\eps$ be the smallest real number such that 
$$
\min\{|\tau-\zeta_3|,|\tau-\zeta_6|\}\le\eps
$$
for every~$\tau$ like this. In particular, ${0<\eps<10^{-3}}$. 


Lemma~\ref{lcat} implies that
$$
\rho \ge 44000\eps^{3}. 
$$
Now let us estimate $h $. Note that for ${\tau=(b+\sqrt\Delta)/2a}$ we have 
$$
\Im(\tau-\zeta_3)=\Im(\tau-\zeta_6)=\frac{|\Delta|^{1/2}}{2a}-\frac{\sqrt3}{2}.
$$
Hence every suitable ${a>1}$ satisfies 
$$
\frac{\sqrt3}{2}<\frac{|\Delta|^{1/2}}{2a}\le \frac{\sqrt3}{2}+\eps, 
$$
which can be rewritten as 
$$
\frac{|\Delta|^{1/2}}{\sqrt3+2\eps}\le a<\frac{|\Delta|^{1/2}}{\sqrt3}.
$$
Since all such~$a$ are prime and each occurs in at most~$2$ triples ${(a,b,c)\in T_\Delta}$, we have 
$$
h  \le 2\left(\pi\left(\frac{|\Delta|^{1/2}}{\sqrt3}\right) - \pi\left(\frac{|\Delta|^{1/2}}{\sqrt3+2\eps}\right)\right)+3.
$$
(We have to add~$3$ rather than~$1$ because ${|\Delta|^{1/2}/(\sqrt3+2\eps)}$ can accidentally be a prime number.)
Using~\eqref{epiappr}, we estimate
\begin{align*}
\pi\left(\frac{|\Delta|^{1/2}}{\sqrt3}\right)&\le \frac{2}{\sqrt3}\frac{|\Delta|^{1/2}}{\log|\Delta|-\log3-3}\\
&< \frac{2}{\sqrt3}\frac{|\Delta|^{1/2}}{\log|\Delta|}+6\frac{|\Delta|^{1/2}}{(\log|\Delta|)^2},\\
\pi\left(\frac{|\Delta|^{1/2}}{\sqrt3+2\eps}\right)&\ge \frac{2}{\sqrt3+2\eps}\frac{|\Delta|^{1/2}}{\log|\Delta|-\log3-1}\\
&> \left(\frac{2}{\sqrt3}-\frac{4\eps}{3}\right)\frac{|\Delta|^{1/2}}{\log|\Delta|}-2\frac{|\Delta|^{1/2}}{(\log|\Delta|)^2},
\end{align*}
which implies that 
\begin{align*}
h &<\frac{8\eps}{3}\frac{|\Delta|^{1/2}}{\log|\Delta|}+9\frac{|\Delta|^{1/2}}{(\log|\Delta|)^2}, \\
h \log\rho &>\left(\frac{8\eps}{3}\frac{|\Delta|^{1/2}}{\log|\Delta|}+9\frac{|\Delta|^{1/2}}{(\log|\Delta|)^2} \right)\log(44000\eps^{3})
\end{align*}
We will estimate each of the terms 
$$
\frac{8\eps}{3}\frac{|\Delta|^{1/2}}{\log|\Delta|}\log(44000\eps^{3}), \qquad 9\frac{|\Delta|^{1/2}}{(\log|\Delta|)^2} \log(44000\eps^{3})
$$
separately.

The function ${\eps\mapsto \eps\log(44000\eps^3)}$ is strictly decreasing on ${(0,10^{-3}]}$, which implies that 
$$
\frac{8\eps}{3}\frac{|\Delta|^{1/2}}{\log|\Delta|}\log(44000\eps^{3})\ge \frac{8\cdot10^{-3}}{3}\frac{|\Delta|^{1/2}}{\log|\Delta|}\log(4.4\cdot10^{-5})>- \frac{|\Delta|^{1/2}}{\log|\Delta|}. 
$$
To estimate the second term, note that, since ${a<|\Delta/3|^{1/2}}$, we have 
$$
\eps\ge\frac{|\Delta|^{1/2}}{2a}-\frac{\sqrt3}{2}=\frac{|\Delta|-3a^2}{2a(|\Delta|^{1/2}+a\sqrt3)}\ge \frac{1}{2a(|\Delta|^{1/2}+a\sqrt3)}>\frac{\sqrt3}{4|\Delta|}.
$$
Hence ${\log(44000\eps^{3})>-3\log|\Delta|}$, and 
$$
9\frac{|\Delta|^{1/2}}{(\log|\Delta|)^2} \log(44000\eps^{3})>-27\frac{|\Delta|^{1/2}}{(\log|\Delta|)}.
$$
This proves that 
$$
h \log\rho >-28\frac{|\Delta|^{1/2}}{\log|\Delta|},
$$
better than wanted. 

\subsection{Proof of Proposition~\ref{pn}}
Let $x_0,x_1,\ldots,x_{h-1}$ be the singular moduli of discriminant~$\Delta$, with~$x_0$ dominant. Then 
\begin{align*}
&\log\left(1-e^{-\pi|\Delta|^{1/2}+\log2079}\right)\le\log|x_0|- \pi|\Delta|^{1/2}\le \log\left(1+e^{-\pi|\Delta|^{1/2}+\log2079}\right),\\
& \log \left(1-e^{-\pi|\Delta|^{1/2}+\log|\Delta|}\right)\le \log|x_i|-\log\rho \le0 \qquad (i=1,\ldots,h-1),
\end{align*}
where the second inequality follows from Corollary~\ref{cprinceqrefined}. 

Summing up, we obtain 
$$
\log N(\Delta) = \pi|\Delta|^{1/2}+(h -1)\log\rho  +O_1\left( h  \left| \log \left(1-e^{-\pi|\Delta|^{1/2}+\log|\Delta|}\right)\right|\right). 
$$
Since ${|\Delta|\ge 10^{11}}$ and  ${h\le 3|\Delta|^{1/2}/\log|\Delta|}$,  we may bound the error term by   ${ e^{-\pi|\Delta|^{1/2}+2\log|\Delta|}}$, which proves~\eqref{enasymp}. And~\eqref{enlower} follows from~\eqref{enasymp} and Proposition~\ref{phlogrho}.

\section{The signature theorem}
\label{ssign}
In this section we prove Theorem~\ref{thsign}. Let us reproduce the statement here. 

\begin{theorem}
\label{thm-ntwo}
Let~$\Delta$ be a trinomial discriminant of signature $(m,n)$.  Assume that ${|\Delta|\ge 10^{40}}$. Then ${m-n\le 2}$. 
\end{theorem}

As before, in this section we assume that~$\Delta$ is a trinomial discriminant. In particular,~$\Delta$ is odd, square-free and has at most~$2$ prime divisors. 
We also use the notation $h,\rho,N$, etc. from Section~\ref{shrhon}.  Throughout this section~$L$ is the Hilbert Class Field of $\Q(\sqrt\Delta)$. It is an unramified abelian extension of $\Q(\sqrt\Delta)$ of degree~$h$, generated over $\Q(\sqrt\Delta)$ by any singular modulus of discriminant~$\Delta$. It is also Galois over~$\Q$, of degree~$2h$. Denoting
$$
H=\Gal(L/\Q(\sqrt\Delta)), \qquad G=\Gal(L/\Q),
$$ 
we have ${G=H\cup H\iota}$, where ${\iota\in G}$ denotes the complex conjugation. Note that ${\sigma\iota=\iota\sigma^{-1}}$ for every ${\sigma \in H}$; see, for instance, \cite[Lemma~9.3]{Co13}. 

We denote by $\norm(\cdot)$ the absolute norm on~$L$; that is,  for ${y\in L}$ we set
$$
\norm(y) = \left|\norm_{L/\Q}(y)\right|= \prod_{\sigma\in G}|y^\sigma| = \prod_{\sigma\in H}|y^\sigma|^2. 
$$
If~$x$ is a singular modulus of discriminant~$\Delta$ then ${\norm(x)=N(\Delta)^2}$. Indeed, ${N(\Delta)=\bigl|\norm_{\Q(x)/\Q}(x)\bigr|}$, and~$\Q(x)$ is a subfield of~$L$ of degree~$2$. 

The strategy is as follows. 
We introduce a certain non-zero algebraic integer ${z\in L}$ and estimate $\norm(z)$ from above using the ``principal inequality'' as given in  Corollary~\ref{cprinceqrefined}. Compared with the trivial lower estimate ${\norm(z)\ge 1}$, this would imply the following weaker version of Theorem~\ref{thm-ntwo}: when $|\Delta|$ is large we have ${m-n\le 4}$.

Using this, and applying Proposition~\ref{pnumroots},  we obtain a non-trivial lower bound for $\norm(z)$. 
Comparing it with the previously obtained upper bound, we prove Theorem~\ref{thm-ntwo}. 

We  start with some lemmas.

\begin{lemma}
\label{lznz}
Let~$\Delta$ be a trinomial discriminant and $x_1,\ldots,x_4$  distinct singular moduli of discriminant~$\Delta$. Then 
${x_1x_2\ne x_3x_4}$. 
\end{lemma}

\begin{proof}
Applying Galois conjugation, we may assume that~$x_1$ is dominant. Then neither of $x_2,x_3,x_4$ is. From Proposition~\ref{pcn&rho} and Lemma~\ref{lbfz}, we have 
$$
|x_2|\ge 700|\Delta|^{-3}, \qquad |x_3|,|x_4|<|\Delta|^{0.8}. 
$$
It follows that 
$$
e^{\pi|\Delta|^{1/2}}-2079 <|x_1|=|x_2^{-1}x_3x_4|< |\Delta|^{4.6}/700,
$$
which is impossible for ${|\Delta|\ge 10^{11}}$. 
\end{proof}

\begin{lemma}
\label{lfour}
A trinomial ${t^m+At^n+B\in \R[t]}$ may have at most~$4$ real roots.
\end{lemma}

\begin{proof}
The derivative ${mt^{m-1}+nAt^{n-1}}$ may have at most~$3$ real roots, and the result follows by the Theorem of Rolle.  

Alternatively, one may use the ``Descartes' rule of signs'', which implies that the trinomial may have at most~$2$ positive and at most~$2$ negative roots. 
\end{proof}


Since 
 ${h(\Delta)\ge 101>6}$ by Proposition~\ref{pcn&rho}, Lemma~\ref{lfour} implies that there must exist at least~$3$ non-real singular moduli of discriminant~$\Delta$. In particular, there exist two singular moduli ${x_1,x_2\notin \R}$ such that ${x_1\ne x_2}$ and ${\bar x_1\ne x_2}$. (We denote by~$\bar x$ the complex conjugate of~$x$.) Thus, ${x_1,\bar x_1,x_2, \bar x_2}$ are distinct non-dominant singular moduli of discriminant~$\Delta$. We set
$$
z=x_1\bar x_1-x_2\bar x_2=|x_1|^2-|x_2|^2. 
$$
This~$z$ is a non-zero (by Lemma~\ref{lznz}) real algebraic integer. 

\begin{proposition}
\label{pnorzupper}
Let~$\Delta$ be a trinomial discriminant of signature $(m,n)$. Then 
$$
\log \norm(z) \le \pi|\Delta|^{1/2}(8-2(m-n))+2(m-n)\log |\Delta| +243|\Delta|^{1/2}/\log|\Delta|
$$
\end{proposition}

\begin{proof}
Using Corollary~\ref{cprinceqrefined}, we estimate 
$$
|z|< 2\rho^2e^{(m-n)(-\pi|\Delta|^{1/2}+\log|\Delta|)}.
$$
Let us  split~$G$ into three subsets.

\begin{enumerate}
\item
For ${\sigma =\id}$ or ${\sigma=\iota}$ we have ${z^\sigma=z}$. Hence in these two cases 
\begin{equation}
\label{etiny}
|z^\sigma|< 2\rho^2e^{(m-n)(-\pi|\Delta|^{1/2}+\log|\Delta|)}. 
\end{equation}

\item
For every singular modulus~$x$ of discriminant~$\Delta$ there exists exactly one element ${\sigma \in H}$ such that~$x^\sigma$ is dominant. We claim that $\bar x^{\sigma^{-1}}$ is then dominant as well. Indeed, since ${x^\sigma\in \R}$ (the dominant singular modulus is real), we have 
${\bar x^{\sigma^{-1}}=x^{\iota\sigma^{-1}}=x^{\sigma\iota}= x^\sigma}$, 
as wanted. 

Now let ${\sigma_1,\sigma_2\in H}$ be such that $x_1^{\sigma_1}$ is dominant, and so is $x_2^{\sigma_2}$. Then there exist exactly~$8$ elements ${\sigma \in G}$ such that one of ${x_1^\sigma,\bar x_1^\sigma, x_2^\sigma, \bar x_2^\sigma}$ is dominant: they are
$$
\sigma_1,\sigma_1\iota, \sigma_1^{-1},\sigma_1^{-1}\iota,\sigma_2,\sigma_2\iota, \sigma_2^{-1},\sigma_2^{-1}\iota. 
$$
For these~$\sigma$ we have the upper estimate 
\begin{equation}
\label{ehuge}
|z^\sigma|< \rho (e^{\pi|\Delta|^{1/2}}+2079)+\rho^2< 2\rho e^{\pi|\Delta|^{1/2}}. 
\end{equation}

\item
For the remaining ${2h-10}$ elements ${\sigma \in G}$ none of ${x_1^\sigma,\bar x_1^\sigma, x_2^\sigma, \bar x_2^\sigma}$ is dominant. Hence for those~$\sigma$ we have 
\begin{equation}
\label{emoder}
|z^\sigma|\le 2\rho^2. 
\end{equation}

\end{enumerate}

It follows that 
\begin{align*}
\norm(z) &< \bigl(2\rho^2e^{(m-n)(-\pi|\Delta|^{1/2}+\log|\Delta|)}\bigr)^2 \bigl(2\rho e^{\pi|\Delta|^{1/2}}\bigr)^8 (2\rho^2)^{2h-10}\\
&\le e^{\pi|\Delta|^{1/2}(8-2(m-n)) +4h\log\rho^\ast +2h\log2+2(m-n)\log |\Delta|},
\end{align*}
where ${\rho^\ast =\max\{\rho,1\}}$. Using Propositions~\ref{pcn&rho} and~\ref{phlogrho}, we estimate 
$$
2h\log\rho^\ast+h\log2 \le 243|\Delta|^{1/2}/\log|\Delta|. 
$$
Whence the result.
\end{proof}

\begin{corollary}
\label{cm-nfour}
We have either ${|\Delta|< 10^{30}}$ or ${m-n\le 4}$.
\end{corollary}

\begin{proof}
Since ${\norm(z)\ge 1}$, Proposition~\ref{pnorzupper} implies that 
$$
(m-n)\left(1-\frac{\log|\Delta|}{\pi|\Delta|^{1/2}}\right)-4 \le \frac{130|\Delta|^{1/2}/\log|\Delta|}{\pi|\Delta|^{1/2}}.
$$
Hence 
$$
m-n<\left(4+\frac{50}{\log|\Delta|}\right)\left(1-\frac{\log|\Delta|}{\pi|\Delta|^{1/2}}\right)^{-1}. 
$$ 
When ${|\Delta|\ge 10^{30}}$ this implies that 
${m-n<4.8}$. 
Hence ${m-n\le 4}$, as wanted. 
\end{proof}

To improve on this, we need a non-trivial lower estimate for $\norm(z)$.

\begin{proposition}
\label{pnorzlower}
Assume that ${|\Delta|\ge 10^{30}}$. Then ${\log\norm(z)>3.36\log N(\Delta)}$. 
\end{proposition}

\begin{proof}
For every rational prime number~$p$ we want to estimate ${\nu_p(\norm(z))}$ from below. 
Let ${\gerp\mid p}$ be an $L$-prime above~$p$ and~$e_p$ the ramification index of~$p$ in~$L$. Note that 
\begin{equation}
\label{eep}
e_p=
\begin{cases}
1, & p\nmid \Delta,\\
2, & p\mid \Delta,
\end{cases}
\end{equation}
because~$L$ is unramified over $\Q(\sqrt\Delta)$. 
We denote by $\nu_\gerp(\cdot)$ the $\gerp$-adic valuation and define ${\nu_\gerp'(\cdot)=\nu_\gerp(\cdot)/e_p}$. Then, clearly,  ${\nu_p(m)=\nu_\gerp'(m)}$ for any ${m\in \Z}$. 

Proposition~\ref{pnapi} implies that the set  
$$
S_p=S_p(\Delta)=\{\nu_\gerp'(x): \text{$x$ singular modulus of discriminant~$\Delta$}\}. 
$$
may consist of at most~$2$ elements. In case it consists of just one element, this element is  ${\nu_\gerp'(N)/h=\nu_p(N)/h}$, where, as before, we use notation ${h=h(\Delta)}$ and ${N=N(\Delta)}$. Hence for every ${\sigma \in G}$ we have 
$$
\nu_\gerp'(x_1^\sigma)=\nu_\gerp'(\bar x_1^\sigma)=\nu_\gerp'(x_2^\sigma)=\nu_\gerp'(\bar x_2^\sigma)=\nu_p(N)/h,
$$
which implies that 
${\nu_\gerp'(z^\sigma)\ge 2\nu_p(N)/h}$. 
Therefore in this case
\begin{equation}
\label{efour}
\nu_p(\norm(z)) =\nu_\gerp'(\norm(z))\ge 2h\cdot 2\nu_p(N)/h =4\nu_p(N). 
\end{equation}

Now assume that~$S_p$ consists of~$2$ distinct elements: ${S_p=\{\alpha_p,\beta_p\}}$, with 
$$
\alpha_p<\nu_p(N)/h<\beta_p.
$$
Since ${m-n\le 4}$, Proposition~\ref{pnumroots} implies that at most~$4$ singular moduli of discriminant~$\Delta$ have $\nu_\gerp'$-valuation~$\alpha_p$. It follows that there exist at most~$32$ elements ${\sigma\in G}$ such that one of ${x_1^\sigma, \bar x_1^\sigma, x_2^\sigma, \bar x_2^\sigma}$ has $\nu_\gerp'$-valuation~$\alpha_p$.  For the remaining ${2h-32}$ elements ${\sigma \in G}$ we have 
$$
\nu_\gerp'(x_1^\sigma)=\nu_\gerp'(\bar x_1^\sigma)=\nu_\gerp'(x_2^\sigma)=\nu_\gerp'(\bar x_2^\sigma)=\beta_p>\nu_p(N)/h,
$$
which implies that for these~$\sigma$ we have
${\nu_\gerp'(z^\sigma)> 2\nu_p(N)/h}$. 
Therefore 
\begin{equation}
\label{ealmostfour}
\nu_p(\norm(z)) > (2h-32)\cdot 2\nu_p(N)/h >3.36\nu_p(N), 
\end{equation}
where the last inequality follows from ${h\ge 101}$. 

Thus, for every~$p$ we have either~\eqref{efour} or~\eqref{ealmostfour}. This proves the wanted lower bound  ${\log\norm(z)>3.36\log N}$. 
\end{proof}


We are now ready to prove Theorem~\ref{thm-ntwo}.

\begin{proof}[Proof of Theorem~\ref{thm-ntwo}]

Combining the lower estimate from  Propositions~\ref{pnorzlower} with~\eqref{enlower}, we obtain
$$
\log\norm(z) \ge 3.36\pi |\Delta|^{1/2} -110|\Delta|^{1/2}/\log|\Delta|. 
$$
Comparing this with the upper estimate from Proposition~\ref{pnorzupper}, we obtain  
$$
\pi|\Delta|^{1/2}(4.64-2(m-n))+8\log |\Delta| +360|\Delta|^{1/2}/\log|\Delta|>0. 
$$
If ${m-n\ge 3}$ then this implies that ${|\Delta|<10^{40}}$ after a trivial calculation. 
\end{proof}

{\footnotesize

\bibliographystyle{amsplain}
\bibliography{trinom}

\providecommand{\bysame}{\leavevmode\hbox to3em{\hrulefill}\thinspace}
\providecommand{\MR}{\relax\ifhmode\unskip\space\fi MR }
\providecommand{\MRhref}[2]{%
  \href{http://www.ams.org/mathscinet-getitem?mr=#1}{#2}
}
\providecommand{\href}[2]{#2}
\begin{thebibliography}{10}

\bibitem{BFZ20}
Yuri Bilu, Bernadette Faye, and Huilin Zhu, \emph{Separating singular moduli
  and the primitive element problem}, Q. J. Math. \textbf{71} (2020), no.~4,
  1253--1280. \MR{4186519}

\bibitem{BHK20}
Yuri Bilu, Philipp Habegger, and Lars K\"{u}hne, \emph{No singular modulus is a
  unit}, Int. Math. Res. Not. IMRN (2020), no.~24, 10005--10041. \MR{4190395}

\bibitem{BL20}
Yuri Bilu and Florian Luca, \emph{Trinomials with given roots}, Indag. Math.
  (N.S.) \textbf{31} (2020), no.~1, 33--42. \MR{4052259}

\bibitem{BLP16}
Yuri Bilu, Florian Luca, and Amalia Pizarro-Madariaga, \emph{Rational products
  of singular moduli}, J. Number Theory \textbf{158} (2016), 397--410.
  \MR{3393559}

\bibitem{BMZ13}
Yuri Bilu, David Masser, and Umberto Zannier, \emph{An effective ``theorem of
  {A}ndr\'e'' for {$CM$}-points on a plane curve}, Math. Proc. Cambridge
  Philos. Soc. \textbf{154} (2013), no.~1, 145--152. \MR{3002589}

\bibitem{Bo06}
Andrew~R. Booker, \emph{Quadratic class numbers and character sums}, Math.
  Comp. \textbf{75} (2006), no.~255, 1481--1492. \MR{2219039}

\bibitem{Bo20}
Matteo Bordignon, \emph{Partial {G}aussian sums and the
  {P}\'{o}lya--{V}inogradov inequality for primitive characters},
  \url{arXiv:2001.05114} (2020).

\bibitem{Bu62}
D.~A. Burgess, \emph{On character sums and {$L$}-series}, Proc. London Math.
  Soc. (3) \textbf{12} (1962), 193--206. \MR{0132733}

\bibitem{Bu63}
\bysame, \emph{On character sums and {$L$}-series. {II}}, Proc. London Math.
  Soc. (3) \textbf{13} (1963), 524--536. \MR{0148626}

\bibitem{Ca86}
J.~W.~S. Cassels, \emph{{L}ocal {F}ields}, London Mathematical Society Student
  Texts, vol.~3, Cambridge University Press, Cambridge, 1986. \MR{861410}

\bibitem{Ch07}
Yong-Gao Chen, \emph{On the {S}iegel-{T}atuzawa-{H}offstein theorem}, Acta
  Arith. \textbf{130} (2007), no.~4, 361--367. \MR{2365711}

\bibitem{Co13}
David~A. Cox, \emph{Primes of the {F}orm {$x^2 + ny^2$}}, second ed., Pure and
  Applied Mathematics (Hoboken), John Wiley \& Sons, Inc., Hoboken, NJ, 2013,
  Fermat, class field theory, and complex multiplication. \MR{3236783}

\bibitem{Hu82}
Loo~Keng Hua, \emph{Introduction to {N}umber {T}heory}, Springer-Verlag,
  Berlin-New York, 1982, Translated from the Chinese by Peter Shiu. \MR{665428}

\bibitem{IR90}
Kenneth Ireland and Michael Rosen, \emph{A classical introduction to modern
  number theory}, second ed., Graduate Texts in Mathematics, vol.~84,
  Springer-Verlag, New York, 1990. \MR{1070716}

\bibitem{IK04}
Henryk Iwaniec and Emmanuel Kowalski, \emph{Analytic {N}umber {T}heory},
  American Mathematical Society Colloquium Publications, vol.~53, American
  Mathematical Society, Providence, RI, 2004. \MR{2061214}

\bibitem{JKL21}
Niraek Jain-Sharma, Tanmay Khale, and Mengzhen Liu, \emph{Explicit burgess
  bound for composite moduli}, Int. J. Number Theory, to appear,
  \url{arXiv:2010.09530}.

\bibitem{LLS15}
Youness Lamzouri, Xiannan Li, and Kannan Soundararajan, \emph{Conditional
  bounds for the least quadratic non-residue and related problems}, Math. Comp.
  \textbf{84} (2015), no.~295, 2391--2412. \MR{3356031}

\bibitem{Li21}
Yingkun Li, \emph{Singular units and isogenies between {CM} elliptic curves},
  Compos. Math. \textbf{157} (2021), no.~5, 1022--1035. \MR{4251608}

\bibitem{LR19}
Florian Luca and Antonin Riffaut, \emph{Linear independence of powers of
  singular moduli of degree three}, Bull. Aust. Math. Soc. \textbf{99} (2019),
  no.~1, 42--50. \MR{3896878}

\bibitem{Lu78}
Edouard Lucas, \emph{Theorie des {F}onctions {N}umeriques {S}implement
  {P}eriodiques}, Amer. J. Math. \textbf{1} (1878), no.~4, 289--321.
  \MR{1505176}

\bibitem{Po17}
Paul Pollack, \emph{Bounds for the first several prime character nonresidues},
  Proc. Amer. Math. Soc. \textbf{145} (2017), no.~7, 2815--2826. \MR{3637932}

\bibitem{Ri19}
Antonin Riffaut, \emph{Equations with powers of singular moduli}, Int. J.
  Number Theory \textbf{15} (2019), no.~3, 445--468. \MR{3925747}

\bibitem{RS62}
J.~Barkley Rosser and Lowell Schoenfeld, \emph{Approximate formulas for some
  functions of prime numbers}, Illinois J. Math. \textbf{6} (1962), 64--94.
  \MR{0137689}

\bibitem{Sc76}
Wolfgang~M. Schmidt, \emph{Equations over {F}inite {F}ields. {A}n {E}lementary
  {A}pproach}, Lecture Notes in Mathematics, Vol. 536, Springer-Verlag,
  Berlin-New York, 1976. \MR{0429733}

\bibitem{Ta51}
Tikao Tatuzawa, \emph{On a theorem of {S}iegel}, Jap. J. Math. \textbf{21}
  (1951), 163--178. \MR{0051262}

\bibitem{pari}
{The PARI Group}, \emph{{PARI/GP} ({V}ersion 2.11.0)}, 2018, Bordeaux,
  \url{http://pari.math.u-bordeaux.fr/}.

\bibitem{sagemath}
{The Sage Developers}, \emph{{S}agemath, the {S}age {M}athematics {S}oftware
  {S}ystem ({V}ersion 8.6)}, 2019, {\tt https://www.sagemath.org}.

\bibitem{Tr15}
Enrique Trevi\~{n}o, \emph{The {B}urgess inequality and the least {$k$}th power
  non-residue}, Int. J. Number Theory \textbf{11} (2015), no.~5, 1653--1678.
  \MR{3376232}

\bibitem{Wa04}
Mark Watkins, \emph{Class numbers of imaginary quadratic fields}, Math. Comp.
  \textbf{73} (2004), no.~246, 907--938. \MR{2031415}

\end{thebibliography}

}

\end{document}